\documentclass[11pt,twoside]{article}
\usepackage{bbm}
\usepackage{mathrsfs}
\usepackage{amssymb}
\usepackage{amsmath}
\usepackage{amsthm}
\usepackage{amsfonts}
\usepackage{color}
\usepackage{graphicx}
\usepackage[active]{srcltx}
\usepackage{CJK}
\usepackage{geometry}
\usepackage[cp1252]{inputenc}

\usepackage{mathrsfs}
\usepackage{graphicx}

\usepackage[active]{srcltx}

\allowdisplaybreaks

\usepackage{titletoc}
\titlecontents{section}[0pt]{\addvspace{2pt}\filright}
              {\contentspush{\thecontentslabel\ }}
              {}{\titlerule*[8pt]{.}\contentspage}

\def\rr{{\mathbb R}}

\def\zz{{\mathbb Z}}
\def\nn{{\mathbb N}}

\def\ch{\mathcal H}

\def\cu{{\mathcal U}}

\def\ct{{\mathcal T}}

\def\ci{\mathcal I}

\def\bfc{\textbf{C}}

\def\fz{\infty}
\def\az{\alpha}

\def\lz{\lambda}
\def\dz{\delta}
\def\bdz{\Delta}
\def\ez{\epsilon}

\def\bz{\beta}

\def\gz{{\gamma}}
\def\oz{{\omega}}

\def\sz{\sigma}

\def\wz{\widetilde}

\def\bint{{\ifinner\rlap{\bf\kern.35em--}
\int\else\rlap{\bf\kern.45em--}\int\fi}\ignorespaces}

\def\bbint{{\ifinner\rlap{\bf\kern.35em--}
\hspace{0.078cm}\int\else\rlap{\bf\kern.45em--}\int\fi}\ignorespaces}

\def\lr{\right}
\def\lf{\left}

\newtheorem{thm}{Theorem}[section]
\newtheorem{lem}[thm]{Lemma}
\newtheorem{rem}[thm]{Remark}
\newtheorem{defn}[thm]{Definition}
\newtheorem{example}[thm]{Example}

\numberwithin{equation}{section}

\begin{document}

\arraycolsep=1pt

\title{\Large\bf
Dimension Estimates on Circular $(s,t)$-Furstenberg Sets
\footnotetext{\hspace{-0.35cm}
\endgraf
 2010 {\it Mathematics Subject Classification:} Primary 28A75; Secondary 28A78 $\cdot$ 28A80
 \endgraf {\it Key words and phrases:}    Furstenberg set, circular Furstenberg set, Hausdorff dimension.
 \endgraf J. L. is supported by the Academy of Finland via the projects: Quantitative rectifiability in Euclidean and non-Euclidean spaces, Grant No. 314172, and Singular integrals, harmonic functions, and boundary regularity in Heisenberg groups,
Grant No. 328846.
}
}
\author{Jiayin Liu}
\date{}
\maketitle

\begin{center}
\begin{minipage}{13.5cm}\small
{\noindent{\bf Abstract.}
In this paper, we show that circular $(s,t)$-Furstenberg sets in $\mathbb R^2$ have Hausdorff dimension at least $$\max\{\frac{t}3+s,(2t+1)s-t\} \text{ for all $0<s,t\le 1$}.$$

This result extends the previous dimension estimates on circular Kakeya sets by Wolff.
}
\end{minipage}
\end{center}


\section{Introduction}
Let $F$ be a circular $(s,t)$-Furstenberg set in $\rr^2$. That is, there exists a \textit{parameter set} $K \subset \rr^3_+$ with Hausdorff dimension
$$  \dim_\ch K \ge t $$
such that for every $(x,r) \in K$,
\begin{equation}\label{as1}
  \dim_\ch(F \cap S(x,r)) \ge s
\end{equation}
where $\rr^3_+:=\{(x,r)=(x_1,x_2,r) \ | \ r>0 \}$ and $S(x,r)$ is the circle centered at $x \in \rr^2$ with radius $r$.
A special class of circular $(1,1)$-Furstenberg sets is the family of circular Kakeya sets, that is, Borel sets in $\rr^2$ that contain circles of every radius.

\medskip

The study on the Hausdorff dimension of Furstenberg sets was initiated from their linear version. In this paper, we call a set $F \subset \rr^2$ a linear $(s,t)$-Furstenberg set if
there exists a parameter set $K$ in $A(2,1)$ with
$$  \dim_\ch K \ge t $$
such that for every $L \in K$,
\begin{equation*}
  \dim_\ch(F \cap L) \ge s
\end{equation*}
where $A(n,k)$ denotes the family of $k$-dimensional affine subspaces in $\rr^n$.

In 1999, Wolff \cite{w99} showed that linear $(s,1)$-Furstenberg sets with parameter set $K$ containing lines in every direction have Hausdorff dimension at least
\begin{equation}\label{wolfflin}
  \max\{ \frac12+s, 2s \}    \mbox{ for all $0<s \le 1$}.
\end{equation}
In the sequel, there is a series of works improving the above lower bound and providing the one for linear $(s,t)$-Furstenberg sets with some of them only considering special values of $s,t$. We refer the readers to \cite{kt01,b03,mr12,ls17,t17,hsy20,bz21,dov21,ts21} and references therein.
Moreover, in higher dimensions, one can similarly define linear $(s,t)$-Furstenberg sets with parameter set $K$ in $A(n,k)$. See \cite{h19,hkm19} for some recent progress.

It is not clear whether the above lower bound estimates on the Hausdorff dimension for linear $(s,t)$-Furstenberg sets in $\rr^2$ are sharp for any value of $s$ and $t$ except $s=1$. Hence determining the sharp lower bound remains open for Hausdorff dimension of linear $(s,t)$-Furstenberg sets.

In terms of circular $(s,t)$-Furstenberg sets in $\rr^2$, Wolff in \cite[Corollary 3]{w00} showed that circular Kakeya sets in $\rr^2$ have full dimension $2$ employing techniques from harmonic analysis. Also, in \cite[Corollary 3]{w97}, Wolff proved that Borel sets in $\rr^2$ consisting of circles with $t$-dimensional set of centers have Hausdorff dimension at least $1+t$.  Later, in \cite{kov17}, as an application of their techniques to prove a Marstrand-type restricted projection theorem, K\"{a}enm\"{a}ki-Orponen-Venieri were able to show that the above lower bound $1+t$ in \cite{w97} holds true for analytic $t$-dimensional family of circles.  Hence they provide an alternative method showing the dimension of sets containing full circles. Since the above results concern special cases of circular $(1,t)$-Furstenberg sets, these bounds are sharp. To the best of the author's knowledge, these works and earlier results on families of full circles are the only ones concerning the Hausdorff dimension for circular Furstenberg sets.

In this paper, we extend the existing result to general $0<s,t\le1$. We show the following:

\begin{thm} \label{main1}
     For any $0<s\le 1$ and $0<t\le 1$,
    the Hausdorff dimension of a circular $(s,t)$-Furstenberg set $F$ in $\rr^2$ is at least
    \begin{equation}\label{bound}
      \max\{\frac{t}3+s,(2s-1)t+s\}.
    \end{equation}
\end{thm}

We remark that for any $0<t \le 1$, if $0<s\le \frac23$, then the maximum in \eqref{bound} is attained by $\frac{t}3+s$. Otherwise, it is achieved by $(2s-1)t+s$.
Indeed, these two bounds are obtained by different approaches. Hence Theorem \ref{main1} is a combination of the following two theorems.

\begin{thm} \label{main2}
     For any $0<s\le 1$ and $0<t\le 1$,
    the Hausdorff dimension of a circular $(s,t)$-Furstenberg set $F$ in $\rr^2$ is at least
    \begin{equation*}
      \frac{t}3+s.
    \end{equation*}
\end{thm}

\begin{thm} \label{main3}
     For any $\frac12<s\le 1$ and $0<t\le 1$,
    the Hausdorff dimension of a circular $(s,t)$-Furstenberg set $F$ in $\rr^2$ is at least
    \begin{equation*}
      (2s-1)t+s.
    \end{equation*}
\end{thm}

Below, we briefly outline our ideas of the proof of Theorem \ref{main2} and Theorem \ref{main3}, which will imply Theorem \ref{main1}.
Here, we will focus on explaining some informal ideas on obtaining the Minkowski dimension lower bounds for circular Furstenberg sets. Then we can derive the Hausdorff dimension lower bounds from the Minkowski dimension lower bounds in a standard way. To this end, in the proof, we
will work with a discretized version of the circular $(s,t)$-Furstenberg set $F$ in the following sense. That is, instead of studying the $t$ dimensional parameter set $K$, we will concentrate on a finite subset $V \subset K$ which is a $(\dz, t)$-set (See Definition \ref{def2}). In brief, $V$ is a $\dz$-separated set with cardinality $\dz^{-t}$ and satisfies a $t$-dimensional non-concentration condition.

With this discretized circular Furstenberg set $\cup_{z \in V} S(z) \cap F$, we consider an arbitrary cover  $\cu = \{ B(x_i,r_i)\}_{i \in\ci_{k_1}}$ of this set
by balls of radii between $\dz/2$ and $\dz$ where $\dz=2^{-k_1}$ ($k_1 \in \nn$) is sufficiently small.
We will give a lower bound of $\#\ci_{k_1}$ independent of the choice of the cover $\cu$.  Recall that the desired lower bound is $\frac{t}3+s$ in  Theorem \ref{main2} and $(2t+1)s-t$ in Theorem \ref{main3},
so we need to show that
\begin{equation}\label{ap1}
  \# \ci_{k_1} \gtrsim  \lf (\frac{1}\dz \lr)^{\frac{t}3+s} \mbox{ in Theorem \ref{main2}}
\end{equation}
and
\begin{equation}\label{ap2}
   \# \ci_{k_1} \gtrsim  \lf (\frac{1}\dz \lr)^{(2t+1)s-t} \mbox{ if $\frac12<s\le 1$ in Theorem \ref{main3}}.
\end{equation}
Indeed, this will imply
\begin{equation*}
  \sum_{i \in \ci_{k_1}} r_i^{\frac{t}3+s} \gtrsim \lf(\frac{1}\dz \lr)^{\frac{t}3+s} \dz^{\frac{t}3+s} \gtrsim 1, \mbox{ in Theorem \ref{main2}}
\end{equation*}
 and
 \begin{equation*}
   \sum_{i \in \ci_{k_1}} r_i^{(2t+1)s-t} \gtrsim   \lf (\frac{1}\dz \lr)^{(2t+1)s-t}  \dz^{(2t+1)s-t} \gtrsim 1 \mbox{ if $\frac12<s\le 1$ in Theorem \ref{main3},}
 \end{equation*}
 which further imply that the $\frac{t}3+s$ (resp. $(2t+1)s-t$) dimensional Hausdorff measure of $F$ is positive and therefore the Hausdorff dimension of $F$ is at least $\frac{t}3+s$ (resp. $(2t+1)s-t$).

To show \eqref{ap1}, we adapt the approach for showing the lower bound for the Hausdorff dimension of linear $(s,1)$-Furstenberg sets used by Wolff in \cite{w99} together with some geometric observations from planar geometry.
The heuristic idea is that, since three points determine a unique circle in the plane provided they are not collinear, we can show that three well-separated $\dz$-balls $B_{i}, \ B_{j}, \ B_{k}$ determine a ``unique'' circle $S(z)$ (not necessarily unique in reality, see the statement before \eqref{inters2}), $z \in V$, with the help of Lemma \ref{geom3}, which intuitively means that there exists a unique circle $S(z)$ with $z \in V$ such that $S(z) \cap B_{l} \ne \emptyset$ for $l=i,j,k$. This further
 enables us to identify the circle $S(z)$ with the triple $(i, j, k)$.
 Indeed, the above manipulations are motivated by Wolff \cite{w99} to show the lower bound $1/2 + s$ in \eqref{wolfflin} for the Hausdorff dimension of linear $(s,1)$-Furstenberg sets where $1/2$ appears from the fact that two points determine a unique line in the plane. For circular $(s,1)$-Furstenberg sets, we can only get the lower bound $1/3 + s$ since we need three points to determine a circle.
 On the other hand, since $S(z) \cap F$ has Hausdorff dimension no less than $s$, we need, roughly speaking, at least $\sim \dz^{-s}$ $\dz$-balls in $\cu$ to cover $S(z) \cap F$. Hence we can identify each $S(z) \cap F$ by the triples $(i, j, k) \in  \ci_{k_1} \times \ci_{k_1} \times \ci_{k_1}$ (or equivalently, $(B_i,B_j,B_k) \in \cu \times \cu \times \cu$) where $S(z) \cap B_{l} \ne \emptyset$ for $l=i,j,k$. Then each $S(z) \cap F$ gives rise to $\dz^{-s} (\dz^{-s}-1)(\dz^{-s}-2) \sim \dz^{-3s}$ many distinct triples $(i, j, k) \in \ci_{k_1} \times \ci_{k_1} \times \ci_{k_1}$ representing three distinct $\dz$-balls in $\cu$ and therefore we obtain a total number $\# V \times \dz^{-3s} = \dz^{-3s-t}$ many distinct triples.
  Finally, since all these triples are contained in $\ci_{k_1} \times \ci_{k_1} \times \ci_{k_1}$, we deduce that $(\#\ci_{k_1})^3 \gtrsim \dz^{-3s-t}$, which gives \eqref{ap1}. This is the rough idea behind the proof of the Minkowski dimesion version of Theorem \ref{main2}.

On the other hand, inequality \eqref{ap2} is obtained by applying the result from K\"{a}enm\"{a}ki-Orponen-Venieri in \cite{kov17} utilised to find  the Hausdorff dimension of $t$-dimensional analytic sets of circles. Heuristically, as discussed above, since one needs at least $\sim \dz^{-s}$ $\dz$-balls in $\cu$ to cover $S(z) \cap F$ for each $z \in V$, if each $\dz$-ball in $\cu$ only intersects one $S(z) \cap F$ for some $z \in V$, then $\cu$ consists of at least $\dz^{-s}\#V \sim \dz^{-s-t}$ many $\dz$-balls. However, this may not be the case. In general, if each $\dz$-ball in $\cu$ intersects no more than $\dz^{-\xi}$ ($0 < \xi \le t$) many sets from the family $\{S(z) \cap F\}_{z \in V}$, then we can deduce that $\cu$ consists of at least $\frac{\dz^{-s}\#V}{\dz^{-\xi}} \sim \frac{\dz^{-s-t}}{\dz^{-\xi}}$ many $\dz$-balls.
Actually, by applying \cite[Lemma 5.1]{kov17}, we can show that for more than half of points $z$ in $V$, there exists $S'(z) \subset S(z) \cap F$ with $\dim_\ch S'(z) = \dim_\ch [S(z) \cap F] \ge s$ such that
each $\dz$-ball in $\cu$ intersects no more than $\dz^{t(2s-2)}$  many sets from the family $\{S'(z)\}_{z \in V}$ where $t(2s-2)$ arises from the choice of the parameter $\lz$ when applying Lemma 5.1 in \cite{kov17} to guarantee \eqref{sim2} holds. We refer readers to the discussion around \eqref{con2} in Section 4 for details. This fact will imply that there exist at least $\frac{\dz^{-s}\#V}{\dz^{t(2s-2)}} \sim \dz^{-[(2t+1)s-t]}$ many $\dz$-balls in $\cu$, which is equivalent to say $\#\ci_{k_1} \gtrsim \dz^{-[(2t+1)s-t]}$. Hence \eqref{ap2} holds and this concludes a heuristic discussion regarding Theorem \ref{main3}.

\medskip
Finally, we remark that we do not know if the bound $\max\{\frac{t}3+s,(2s-1)t+s\}$ in Theorem \ref{main1} is sharp and we here make a conjecture that the sharp lower bound for Hausdorff dimension of circular $(s,1)$-Furstenberg sets is $\frac12+ \frac32s$ for $0<s\le 1$. Indeed, in the following example, based on the example in \cite{w99}, we construct a circular $(s,1)$-Furstenberg set whose Hausdorff dimension does not exceed $\frac12+ \frac32s$  for all $0<s\le 1$.

\begin{example} \rm
     Due to the construction in \cite[Section 1]{w99} by Wolff, for all $0<s\le 1$, there exists a linear $(s,1)$-Furstenberg set $F \subset B(0,4) \setminus B(0,1)$ whose Hausdorff dimension does not exceed $\frac12+ \frac32s$. Now considering $\rr^2$ as the complex plane $\mathbb C$, using the map $\omega : \mathbb C \to \mathbb C, \ z \mapsto \frac{1}{z}$, all lines in $\mathbb C$ are mapped to circles through $(0,0)$. Also noticing that $\oz |_{B(0,4)\setminus B(0,1)  }$ is a biLipschitz homeomorphism, we deduce that $F':=\oz(F)$ is a circular $(s,1)$-Furstenberg set with same dimension as $F$. That is, dim$_\ch (F') \le \frac12+ \frac32s$.
\end{example}

The paper is organised as follows. In Section 2, we clarify our notations and symbols, as well as introduce definitions and results employed in the proof. Section 3 and 4 are devoted to showing the proof of Theorem \ref{main2} and \ref{main3} respectively. In the last section, Section 5, we complete the proof of some auxiliary lemmas needed in the proof of Theorem \ref{main2} using planar geometry.

\bigskip
\begin{center}

\textsc{Acknowledgement
}
\end{center}
J. L. would like to thank K. F\"{a}ssler and T. Orponen for many motivating discussions and their constant support.
J. L. would also like to convey his gratitude to the anonymous referee for pointing out a mistake in the proof of Theorem \ref{main2} and for providing many valuable suggestions which significantly improved the final presentation of the paper.

\section{Preliminaries}
In this paper, we denote by $S^{\dz}(x,r)$ the $\dz$-neighbourhood of $S(x,r)$, i.e.
$$  S^{\dz}(x,r):=   B(x,r+\dz) \setminus B(x,r-\dz)  .$$
We also use the notation $z=(x,r) \in \rr^3$.
Moreover, we use the notation $f \lesssim g$ (resp. $f \lesssim_h g$)
for $f \le k g$ (resp. $f \le k(h) g$) where $k$ is a constant that depends only on the ambient space (resp. the parameter $h$), and may change from line to line. Likewise, $f \gtrsim g$ and $f \sim g$ are understood correspondingly.

The notation $\ch^s$ stands for the $s$-dimensional Hausdorff measure, and $\ch_\fz^s$ stands for
$s$-dimensional Hausdorff content. The notation $|\cdot|$ and $\|\cdot\|$ will denote
the Lebesgue measure and  the Euclidean distance respectively in $\rr^2$ or $\rr^3$. We also use $\text{dist}(A, B)$ to denote Euclidean distance between $A$ and $B$ where $A$ and $B$ can be either points or sets. $\# A$ will denote the cardinality of a set $A$.

We have the following observation which makes it possible to restrict ourselves to circular Furstenberg sets with bounded parameter set.
\begin{rem} \label{rem1}
  \rm (i) Since we are concerned with the Hausdorff dimension of the circular Furstenberg set $F$, we claim that it is enough to consider the case that $F$ has parameter set $K \subset \textbf{B}_0$ where
  \begin{equation}\label{b0}
    \textbf{B}_0= \{(x,r) \in \rr^3 \ | \ x\in B(0,\frac14) \text{ and } \frac12 \le r \le 2\}.
  \end{equation}
  To see this, consider the following covering of the parameter space $\rr^3_+$. For $k,l,m \in \zz$, let
  $$  D_{k,l,m}:= \{(x,r) \in \rr^3 \ | \ x\in B((2^{2m-2}k,2^{2m-2}l),2^{2m-2}) \text{ and }  2^{2m-1} \le r \le 2^{2m+1} \}.$$
  Then
  $$  \rr^3_+ = \bigcup_{k,l,m} D_{k,l,m}$$
  and
  $$\textbf{B}_0 = D_{0,0,0}.$$

  Hence for each $\ez>0$ sufficiently small, there exists $k_\ez, l_\ez, m_\ez$ such that
  \begin{equation}\label{rem02}
    \dim_\ch(K)- \dim_\ch(K \cap D_{k_\ez, l_\ez, m_\ez}) < \ez.
  \end{equation}
   Let $F_{\ez}$ be the circular Furstenberg set with parameter set $ K \cap D_{k_\ez, l_\ez, m_\ez}$.
  Denote by $\mathcal{S}_{y}:\rr^2 \to \rr^2$, $\mathcal{S}_{y}(x):=x-y$ for any $y \in \rr^2$ and by $\mathcal{D}_\lz:\rr^2 \to \rr^2$, $\mathcal{D}_\lz(x):=\lz x$ for any $\lz>0$.

  Then, letting $y=(2^{2m_\ez-2}k_\ez,2^{2m_\ez-2}l_\ez)$ and $\lz= 2^{-2m_\ez}$, we observe that
  $$\wz F_\ez := \mathcal{D}_{2^{-2m_\ez}} \circ \mathcal{S}_{(2^{2m_\ez-2}k_\ez,2^{2m_\ez-2}l_\ez)}(F_\ez)$$
  is a circular Furstenberg set the parameter set $\wz K_\ez$  contained in $\textbf{B}_{0}$ and satisfying
  \begin{equation}\label{rem01}
    \dim_\ch(\wz K_\ez)=\dim_\ch(K \cap D_{k_\ez, l_\ez, m_\ez}).
    \end{equation}
    If $F$ is a circular $(s,t)$-Furstenberg set, then by \eqref{rem02} and \eqref{rem01}, for $0<\ez<t$, we know
    $\wz F_\ez$ is a circular $(s,t-\ez)$-Furstenberg set and
    \begin{equation}\label{rem03}
      \dim_\ch F \ge \dim_{\ch} \wz F_\ez   \text{ for every $0<\ez<t$}.
    \end{equation}
  Now, assume Theorem \ref{main1} holds for circular Furstenberg sets with parameter set contained in $\textbf{B}_{0}$,
  then
  \begin{equation}\label{rem04}
    \dim_{\ch} \wz F_\ez \ge \max\{\frac{t-\ez}3+s,(2s-1)(t-\ez)+s\} \text{ for every $0<\ez<t$}.
  \end{equation}
  Combining \eqref{rem03} and \eqref{rem04},
  we deduce that
  $$ \dim_\ch F \ge \lim_{\ez \to 0} \max\{\frac{t-\ez}3+s,(2s-1)(t-\ez)+s\} = \max\{\frac{t}3+s,(2s-1)t+s\}.$$
  Hence to show Theorem \ref{main1}, we only need to consider the case that $F$ has parameter set $K \subset \textbf{B}_0$.

  \medskip
  (ii) Note that $|S^{\dz}(x,r)| \le c_0 \dz$ for all $(x,r) \in \textbf{B}_0$ where $c_0$ is an absolute constant.
\end{rem}

\medskip

We introduce the following:
\begin{defn}[$(\dz,q)$-sets] \label{def2}
Let $\dz\in (0,1),q>0$, and let $P \subset \rr^n$ be a finite $\dz$-separated set. We say that $P$ is a $(\dz,q)$-set, if it satisfies the estimate
\begin{equation}\label{def1}
  \#\{P \cap B(x,r)\} \lesssim \lf (\frac{r}\dz \lr)^q, \quad x\in \rr^n, \ r>\dz.
\end{equation}
\end{defn}

We recall from \cite[Lemma 3.13]{fo14} the following

\begin{lem} \label{frost}
    Let $\dz,q>0$, and let $Q\subset  \rr^n$ be any set with $\ch^q_\fz(Q)=:\bz>0$. Then there exists a $(\dz,q)$-set $P\subset Q$ with cardinality $\# P \gtrsim \bz \cdot \dz^{-q}$.
\end{lem}

\begin{rem}\label{remsim}
  \rm  If $Q \subset \textbf{B}_0$ and $\ch^q_\fz(Q)=\bz$, by Lemma \ref{frost}, we know that for
any $\dz>0$, there exists a $(\dz,q)$-set $P \subset Q$ with
cardinality $\# P \gtrsim \bz \dz^{-q}$. Furthermore,
letting $r=\mbox{diam} \textbf{B}_0$ in \eqref{def1}, we know $\# P \lesssim
\dz^{-q}$, if $\dz< \text{diam}\textbf{B}_0$.
We conclude that
$$ \bz \dz^{-q} \lesssim \# P  \lesssim
 \dz^{-q}. $$
\end{rem}

To show Theorem \ref{main2}, we need to establish the following result from planar geometry. Since the proof relies on two more auxiliary lemmas, we postpone it to the last section.

\begin{lem} \label{geom3}
  Let $A,B,C \in \rr^2$ such that $\min\{\|A-B\|, \|A-C\|, \|B-C\|,2\} \ge 2c$. For $a>0$ such that $a < \frac1{20}c^2$, define
  \begin{equation}\label{geo6}
W:= \left \{ \begin{array}{ll}
  \quad & \quad b-a \le \|x-A\| \le b+a,  \\
 (x,b) \in  \rr^2 \times [\frac12,2]:  &  \quad b-a \le \|x-B\| \le b+a, \\
   \quad  & \quad b-a \le \|x-C\| \le b+a 
\end{array}
\right \}.
\end{equation}
Then
\begin{equation}\label{geo23}
  \text{\rm diam} \, W \lesssim \frac{a}{c^2}.
\end{equation}
\end{lem}

It is worth mentioning that Lemma \ref{geom3} shares a very similar conclusion with the one in \cite[Lemma 3.2 (Mastrand's 3-circle lemma)]{w99}. Indeed, if we let $\ez=\dz=a$, $r=b$, $\lz=c$, $t=1/2-a$ and $r_1=r_2=r_3=a$ therein, then the set $W$ in Lemma \ref{geom3} will be contained in $\Omega_{\ez t \lz}$ defined in \cite[Lemma 3.2]{w99}. And the conclusion of \cite[Lemma 3.2]{w99} says that $\Omega_{\ez t \lz}$ is contained in the union of two ellipsoids in $\rr^3$ with diam$\, \Omega_{\ez t \lz} \lesssim \frac{a}{c^2}$. Since we only consider the case $r_1=r_2=r_3=a$ (that is, $C_\dz(x_i,r_i)$ become balls $B(x_i,2a)$ for $i=1,2,3$ in \cite[Lemma 3.2]{w99}), we can deduce that $W$ lies in one cuboid in $\rr^3$ based on
an approach which differs completely from the one of \cite[Lemma 3.2]{w99}.

Now, we start the preparation for the proof of Theorem \ref{main3}.
Let $P \subset \rr^3$ be a $(\dz,q)$-set. For any $p \in P$, let $\bdz_p$ be the Dirac measure centered at $p$.
Then
\begin{equation}\label{defmu}
  \mu_P:= \frac{1}{\# P} \sum_{p\in P} \bdz_{p}
\end{equation}
is a probability measure satisfying the Frostman condition $\mu_P(B(z,r)) \lesssim  r^q$ for all $z \in \rr^3$ and $r>\dz$. Indeed, for any ball $B(z,r)$ with $r>\dz$ we have
$$ \mu_P(B(z,r))= \frac{1}{\# P} \lf(\sum_{p\in P} \bdz_{p} \lr)( B(z,r)) = \frac{1}{\# P} \sum_{p\in P} \bdz_{p}(B(z,r))=\frac{1}{\# P} \#(P \cap B(z,r)) \lesssim  r^q.$$

Below in Section 3 and 4, thanks to Remark \ref{rem1}(i), we will assume the circular $(s,t)$-Furstenberg set $F$ has parameter set $K \subset \textbf{B}_0$.

\section{Proof of Theorem 1.2}

\begin{proof}[Proof of Theorem \ref{main2}.]
   Let $F$ be a circular $(s,t)$-Furstenberg set with parameter set $K \subset \textbf{B}_0$. It suffices to show, for any $\ez>0$, $0<s'<s$ and $0<t'<t$,
$$ \dim_\ch(F) \ge \frac{t'}3+s'- \ez.$$
Hence in the following, we fix $s',t'$ and $0<\ez<\frac{t'}3+s'$.

We notice that there exists $\az>0$ and $K_1 \subset K$ such that $\ch^{t'}_\fz(K_1)> \az$,
where
\begin{equation}\label{lbdd11}
  K_1 := \{ z\in K \ | \ \ch^{s'}_\fz(F \cap S(z))> \az\}.
\end{equation}
Indeed, by the subadditivity of Hausdorff content, and the fact
$$ K= \bigcup_n \{ z\in K \ | \ \ch^{s'}_\fz(F \cap S(z))> \frac1{n}\},$$
we deduce the existence of $\az$ such that $\ch^{t'}_\fz(K_1)> \az$ for $K_1$ defined as in \eqref{lbdd11}.

Next, since $\ez>0$, we can find $\dz_0=\dz_0(\ez,s')>0$ sufficiently small such that for any $0<\dz<\dz_0$, we have
\begin{equation}\label{para11}
  \dz^{-\ez}(\log\frac1{\dz})^{-(\frac83+\frac{12}{s'})}
    > 1.
\end{equation}
and
\begin{equation}\label{para12}
  \sqrt{640 \dz} < \tau=\tau(\dz) : =\pi^{-1} (\frac{1}{16})^{1/s'}(\frac{1}{\log \frac{1}{\dz}})^{2/s'} <1.
\end{equation}
Then we choose $k_0$ to be an integer larger than $\log(\frac1{\dz_0})$ also satisfying
\begin{equation}\label{para31}
  \az>\sum_{k=k_0}^{\fz} \frac{1}{k^2}.
\end{equation}


Now, we outline the main steps of the proof. We start with an arbitrary cover  $\cu = \{ B(x_i,r_i)\}_{i \in\ci}$ of $F$
by balls of radius less than $2^{-k_0}$. In the sequel, we will derive a lower bound
$$ \sum_{i \in \ci} r_i^\sz \gtrsim_{\ez,t',s'} 1$$
with $\sz= t'/3+s'- \ez$ independent of the choice of the particular cover. This will imply
$$ \ch^\sz(F) > 0.$$

To this end, we divide the proof into 4 steps.
Let
$$ \ci_k:= \{i \in \ci \ | \  2^{-(k+1)}<r_i \le 2^{-k} \}, \ F_k:= \{\bigcup B(x_i,r_i) \ | \ i \in \ci_k \}.$$

First, in  \textit{Step 1}, we will deduce that there exists $k_1 \ge k_0$ and a $(\dz,t')$-set $V \subset K$ with $\dz=2^{-k_1}$ such that for every circle $z=(x,r) \in V$, we have
\begin{equation}\label{lbdd31}
  \ch^{s'}_\fz(S(z)\cap F_{k_1})> k_1^{-2}.
\end{equation}

Then, in  \textit{Step 2}, we modify Wolff's approach  for linear $(s,1)$-Furstenberg sets to fit our circular case. For each circle $S(z)$ with $z \in V$, we will extract from $S(z)$ three $\tau $-separated arcs $h_z^+,h_z^-,h_z^\times$ such that
\begin{equation}\label{tri10}
  \ch^{s'}_\fz(h_z^+ \cap F_{k_1}) \gtrsim k_1^{-2}, \ \ch^{s'}_\fz(h_z^-\cap F_{k_1})\gtrsim k_1^{-2}, \ \ch^{s'}_\fz(h_z^\times \cap F_{k_1}) \gtrsim  k_1^{-2}.
\end{equation}

These arcs enable us to define an index set $\ct \subset \ci_{k_1} \times \ci_{k_1} \times  \ci_{k_1} \times V$ whose cardinality will be estimated in the following steps and will imply the lower bound for $\#\ci_{k_1}$.

Next, in \textit{Step 3},
we will deduce that the cardinality of $\ct$ is upper bounded by the cardinality of $ \ci_{k_1} $ with the help of Lemma \ref{geom3}. Indeed, we will show
$$\# \ct \lesssim  (\# \ci_{k_1})^3 \tau^{-6}.$$

Finally, in \textit{Step 4}, we will estimate the lower bound of $\# \ct$ which also serves as the one of $\# \ci_{k_1}$, hence $\# \ci$ with the aid of \eqref{tri10}.
This will enable us to conclude the proof.

\bigskip
\textit{Step 1.}
Let $\az$ be as in \eqref{para31}.
Hence by pigeonhole principle we deduce that for each $S(z) \in K_1$, there exists $k(z) \ge k_0$ such that $\ch^{s'}_\fz( S(z)\cap F\cap F_{k(z)})> k(z)^{-2}$.

Moreover, by applying pigeonhole principle again we obtain that there exists $k_1 \ge k_0$ such that
\begin{equation}\label{lbdd21}
  \ch^{t'}_\fz(K_2)> k_1^{-2}
\end{equation}
 where $ K_2 := \{z \in K_1 \ | \ k(z) = k_1\}.$

We remark that for every circle $z \in K_2$, we have
\begin{equation}\label{lbdd41}
  \infty > \ch^{s'}_\fz(S(z)\cap F_{k_1})\ge \ch^{s'}_\fz(S(z)\cap F \cap F_{k_1})> k_1^{-2}.
\end{equation}

By letting $\dz=2^{-k_1}$, $q=t'$ and $Q=K_2$ in Lemma \ref{frost}, we know that there exists a $(\dz,t')$-set $V \subset K_2$ with cardinality
\begin{equation}\label{bd11}
  \# V \gtrsim  \ch^{t'}_\fz(K_2) \cdot \dz^{-t'}.
\end{equation}
Hence for every $z \in V$, \eqref{lbdd41} implies \eqref{lbdd31}, which concludes \textit{Step 1.}


\bigskip
\textit{Step 2.}
We start the procedure of extracting three disjoint arcs for any $S(z), \ z=(x,r) \in V$, which is illustrated in Figure \ref{f0301}, \ref{f0302} and \ref{f0303}. Let
$$  \eta:=\eta(z)= \ch^{s'}_\fz(S(z)\cap F_{k_1}). $$
Also let $\gz=(\frac{\eta}{16})^{1/s'}$.
Divide $S(z)$ into $N$ arcs $I_1, \cdots, I_N$ such that

\begin{itemize}
  \item the length of $I_1 , \cdots, I_{N-1}$ is $\gz$,
  \item the length of $I_N$ is at most $\gz$,
  \item and $N\gz \ge 2\pi r$.
\end{itemize}

Since $\gz=(\frac{\eta}{16})^{1/s'} \le \frac{1}{16}$ and $z=(x,r)\in \textbf{B}_0$ implies $r>\frac12$, we know
$$  N \ge \frac{2\pi r}{\gz} \ge \frac{\pi }{\frac{1}{16}} \ge 16.  $$
Note that if $I$ is an arc in $S(z)$, then
\begin{equation}\label{cont1}
  \ch^{s'}_\fz(I) \le (\text{diam} I)^{s'} \le (\ch^1(I))^{s'}.
\end{equation}
This implies for all $l=1,\cdots, N$,
\begin{equation}\label{cont2}
  \ch^{s'}_\fz(I_l \cap F_{k_1}) \le \ch^{s'}_\fz(I_l) \le \gz^{s'} = \frac{\eta}{16} .
\end{equation}
See Figure \ref{f0301} for $N$ arcs.
\vspace*{7pt}
\begin{figure}[h]
\hspace{-0.7cm}
\centering
\includegraphics[width=10cm]{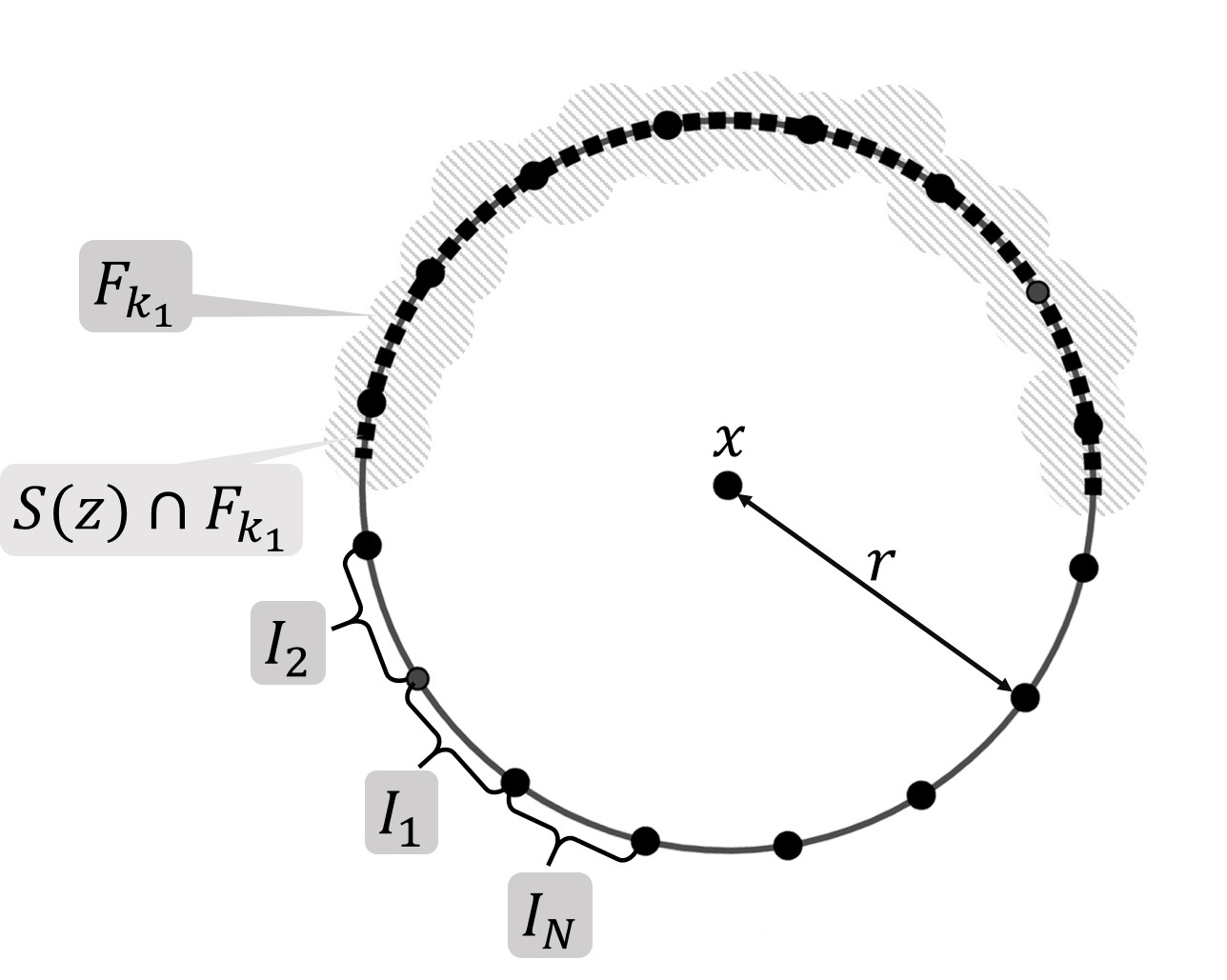}
\caption  {$N$ arcs on $S(z)$.}
\label{f0301}
\end{figure}

Since
\begin{align*}
  \eta & = \ch^{s'}_\fz(S(z) \cap F_{k_1})  \\
   & = \ch^{s'}_\fz(\cup_{l=1}^N I_l \cap F_{k_1})  \\
   & \le \ch^{s'}_\fz(\cup_{l=1}^{N-12} I_l \cap F_{k_1}) +
    \sum_{l=N-11}^{N}\ch^{s'}_\fz( I_l \cap F_{k_1}) \\
    & \le \ch^{s'}_\fz(\cup_{l=1}^{N-12} I_l \cap F_{k_1}) +
     12\frac{\eta}{16}
\end{align*}
where in the last inequality we use \eqref{cont2},
we obtain
$$  \ch^{s'}_\fz(\cup_{l=1}^{N-12} I_l \cap F_{k_1}) \ge \frac14 \eta. $$
This guarantees that there exists $N_1 \in [2,N-12]$ which is the smallest integer satisfying
\begin{equation}\label{cont3}
  \ch^{s'}_\fz(\cup_{l=1}^{N_1} I_l \cap F_{k_1}) \ge \frac18 \eta
\end{equation}
and
\begin{equation}\label{cont4}
  \ch^{s'}_\fz(\cup_{l=1}^{N_1-1} I_l \cap F_{k_1}) < \frac18 \eta .
\end{equation}
Let $h_z^+ := \cup_{l=1}^{N_1} I_l$. By \eqref{cont3} and \eqref{cont4}, we know
\begin{align}\label{eqv1}
  \frac18 \eta & \le \ch^{s'}_\fz (h_z^+) \notag \\
   & \le \ch^{s'}_\fz(\cup_{l=1}^{N_1-1} I_l \cap F_{k_1}) + \ch^{s'}_\fz( I_{N_1} \cap F_{k_1}) \notag \\
   & \le \frac18 \eta + \frac1{16} \eta = \frac3{16} \eta.
\end{align}
See Figure \ref{f0302} for the construction of $h_z^+$.
\vspace*{7pt}
\begin{figure}[h]
\hspace{-0.7cm}
\centering
\includegraphics[width=6cm]{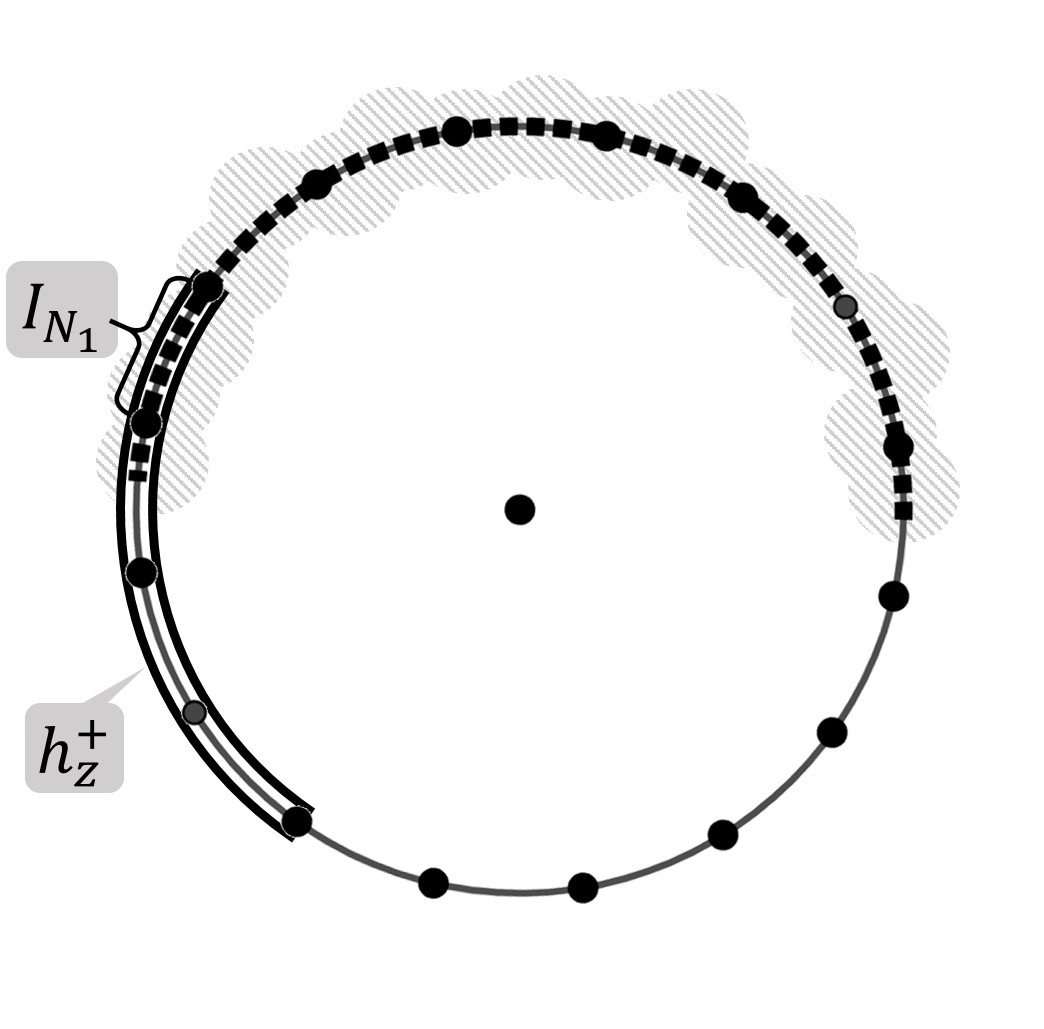}
\caption  {The construction of $h_z^+$.}
\label{f0302}
\end{figure}

Hence the arc $h_z^+$ satisfies the first inequality in \eqref{tri10}.  We continue to construct the other two arcs. Notice that
\begin{equation}\label{content1}
  \ch^{s'}_\fz(\cup_{l=1}^{N_1+1} I_l \cap F_{k_1}) \le  \ch^{s'}_\fz (h_z^+) +  \ch^{s'}_\fz( I_{N_1 +1} \cap F_{k_1}) \le \frac3{16} \eta + \frac1{16} \eta = \frac14 \eta.
\end{equation}
We remark that since $N_1 \le N-12$, we know $N_1+1 \le N-11$. Combining this, \eqref{content1} and \eqref{cont2}, we have
\begin{align*}
  \eta
   & = \ch^{s'}_\fz(\cup_{l=1}^N I_l \cap F_{k_1})  \\
   & \le \ch^{s'}_\fz(\cup_{l=1}^{N_1 +1} I_l \cap F_{k_1}) +
\ch^{s'}_\fz(\cup_{l=N_1 +2}^{N-8} I_l \cap F_{k_1})
    +  \sum_{l=N-7}^{N}\ch^{s'}_\fz( I_l \cap F_{k_1})  \\
    & \le \frac14 \eta + \ch^{s'}_\fz(\cup_{l=N_1 +2}^{N-8} I_l \cap F_{k_1}) +
     8\frac{\eta}{16},
\end{align*}
which implies
$$ \ch^{s'}_\fz(\cup_{l=N_1 +2}^{N-8} I_l \cap F_{k_1})  \ge  \frac{1}{4}\eta.$$
Hence we can find $N_2 \in [N_1 +3,N-8]$ which is the smallest integer satisfying
\begin{equation}\label{cont31}
  \ch^{s'}_\fz(\cup_{l=N_1 +2}^{N_2} I_l \cap F_{k_1}) \ge \frac18 \eta
\end{equation}
and
\begin{equation}\label{cont41}
  \ch^{s'}_\fz(\cup_{l=N_1 +2}^{N_2-1} I_l \cap F_{k_1}) < \frac18 \eta .
\end{equation}
Let $h_z^- := \cup_{l=N_1 +2}^{N_2} I_l$. By \eqref{cont31} and \eqref{cont41}, we know
\begin{align}\label{eqv2}
  \frac18 \eta & \le \ch^{s'}_\fz (h_z^-) \notag \\
   & \le \ch^{s'}_\fz(\cup_{l=N_1 +2}^{N_2 -1} I_l \cap F_{k_1}) + \ch^{s'}_\fz( I_{N_2} \cap F_{k_1}) \notag \\
   & \le \frac18 \eta + \frac1{16} \eta = \frac3{16} \eta.
\end{align}
The construction of the third arc $h_z^\times \subset S(z)$ is similar. See Figure \ref{f0303} for an illustration.
\vspace*{7pt}
\begin{figure}[h]
\hspace{-0.7cm}
\centering
\includegraphics[width=8cm]{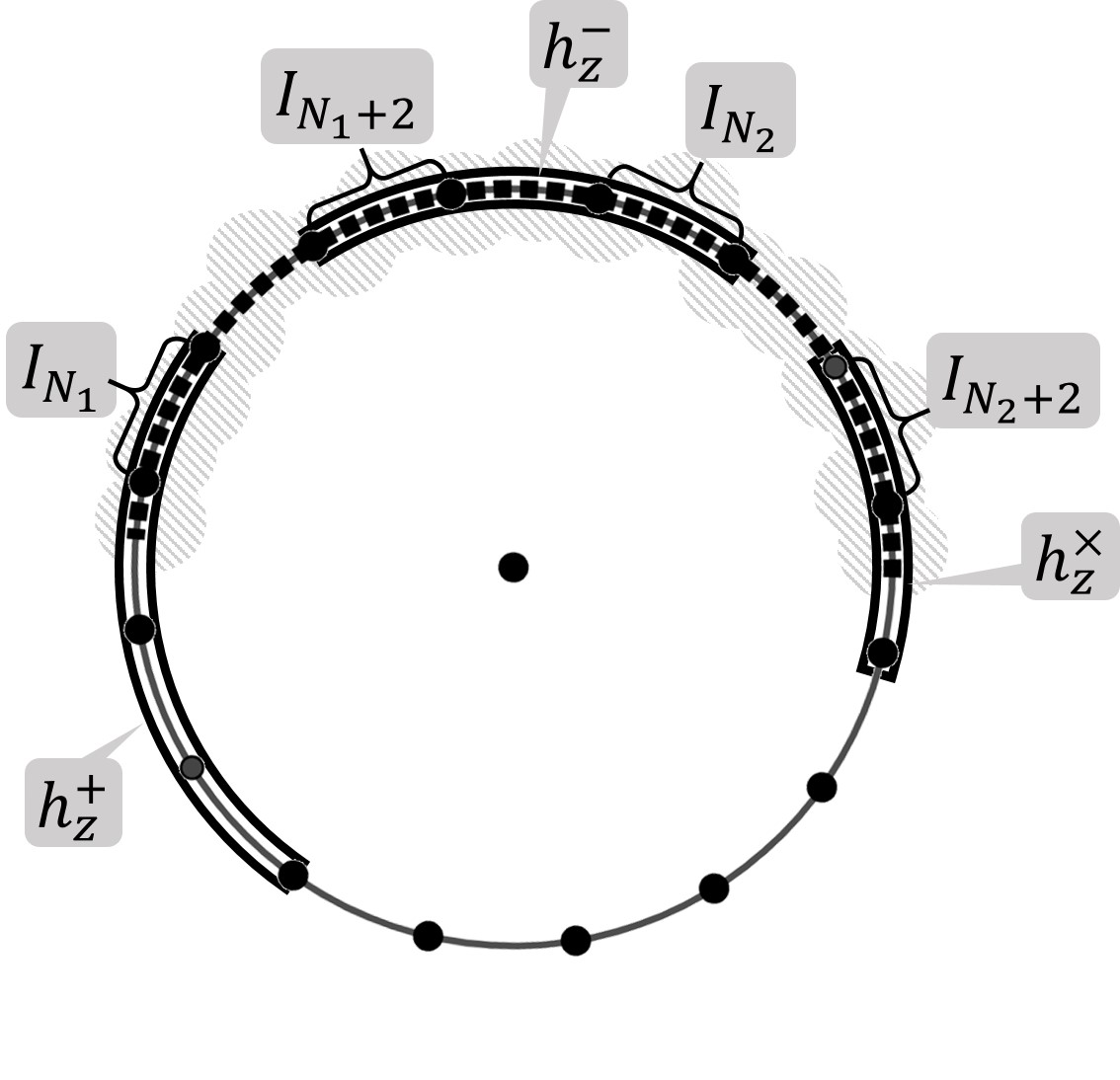}
\caption  {The construction of $h_z^-$ and $h_z^\times$.}
\label{f0303}
\end{figure}

That is, we can find $h_z^\times = \cup_{l=N_2 +2}^{N_3} I_l$ for some integer $N_3 \in [N_2 +3, N-2]$ such that
\begin{align}\label{eqv3}
  \frac18 \eta & \le \ch^{s'}_\fz (h_z^\times) \le \frac3{16} \eta.
\end{align}
We omit the details here. By the construction, it is clear that
$$ \text{dist}(h_z^+,h_z^-)= \text{diam} I_{N_1+1}, \ \text{dist}(h_z^-,h_z^\times)=  \text{diam} I_{N_2+1}, \ \text{and }  \text{dist}(h_z^+,h_z^\times)\ge  \text{diam} I_{N_3+1} .$$
Recall for any $1 \le l \le N-1$, $\ch^1(I_l)=\gz$. Hence $\text{diam} I_{l} \ge \pi^{-1} \gz$ for any $1 \le l \le N-1$.
We conclude that
$$ \min\{\text{dist}(h_z^+,h_z^-), \text{dist}(h_z^-,h_z^\times), \text{dist}(h_z^+,h_z^\times)\} \ge  \pi^{-1} \gz .$$

Therefore, for each circle $S(z)$, we have found three $\pi^{-1}\gz$-separated arcs $h^+_z,h^-_z,h^\times_z \subset S(z)$ with the property in \eqref{eqv1}, \eqref{eqv2} and \eqref{eqv3} respectively. Furthermore, recalling $\gz=(\frac{\eta}{16})^{1/s'}$ and $ \eta > \frac{1}{k_1^2} = \frac{1}{(\log \frac{1}{\dz})^2}$, we deduce that $h^+_z,h^-_z,h^\times_z $ are $\tau=\pi^{-1} (\frac{1}{16})^{1/s'}(\frac{1}{\log \frac{1}{\dz}})^{2/s'} $-separated. Hence by combining with \eqref{lbdd41}, we have showed that \eqref{tri10} holds.

We end \textit{Step 2} by defining
\begin{equation*}
\ct:= \left \{ \begin{array}{ll}
  \quad & \quad h^+_{z} \cap F_{k_1} \cap B_{i_+} \ne \emptyset,  \\
 (i_+,i_-,i_\times,z) \in \ci_{k_1} \times \ci_{k_1} \times  \ci_{k_1} \times V:  &  \quad h^-_{z} \cap F_{k_1} \cap B_{i_-}\ne \emptyset , \\
   \quad  & \quad h^\times_{z} \cap F_{k_1} \cap B_{i_\times} \ne \emptyset
\end{array}
\right \}
\end{equation*}
where $B_{i_+}=B(x_{i_+},r_{i_+})$, $B_{i_-}=B(x_{i_-},r_{i_-})$, and $B_{i_\times}=B(x_{i_\times},r_{i_\times})$. In the following, we will write $x_+$ instead of $x_{i_+}$ for short and other lower indices will be abbreviated correspondingly.

\bigskip
\textit{Step 3.}
We estimate $\# \ct$ from above.

First we fix $i_+,i_-,i_\times$ and estimate the upper bound of the number of $z \in V$ such that $(i_+,i_-,i_\times,z) \in \ct$, where $V$ is chosen as explained above \eqref{bd11}.

To this end, we observe that a necessary condition for $(i_+,i_-,i_\times,z) \in \ct$ is that
\begin{equation}\label{inters}
  S(z) \cap B_{i_+} \ne \emptyset, \ S(z) \cap B_{i_-} \ne \emptyset, \ S(z) \cap B_{i_\times} \ne \emptyset
\end{equation}
and
\begin{equation}\label{inters01}
  \min\{ \|x_+-x_-\|, \|x_+-x_\times \|, \|x_--x_\times\|  \} \ge \tau-2\dz = \tau-\sqrt{4\dz} > \frac{\tau}{2}
\end{equation}
since $h^+_z,h^-_z,h^\times_z$ are $\tau$-separated and $B_{i_+}=B(x_+,r_+)$, $B_{i_-}=B(x_-,r_-)$, $B_{i_\times}=B(x_\times,r_\times)$ are balls of radius between $\dz/2$ and $\dz$. Moreover, in the last inequality of \eqref{inters01} we recall \eqref{para12}.

Hence we will provide an upper bound of $z$ satisfying \eqref{inters} and \eqref{inters01} in the following.
 Assume for some $z = (x,r) \in V$, \eqref{inters} holds. Then we know
\begin{equation*} 
   r - \dz \le \|x-x_+\|  \le r +  \dz, \
   r - \dz \le \|x-x_-\|  \le r +  \dz, \
   r - \dz \le \|x-x_\times\|  \le r +  \dz,
\end{equation*}
which implies

\begin{equation*}
 (x,r) \in  \Gamma:= \left \{ \begin{array}{ll}
  \quad & \quad d-\dz \le \|y-x_+\| \le d+\dz,  \\
 (y,d) \in  \rr^2\times [\frac12,2] :  &  \quad d-\dz \le \|y-x_-\| \le d+\dz, \\
   \quad  & \quad d-\dz \le \|y-x_\times\| \le d+\dz 
\end{array}
\right \}
\end{equation*}
by the fact that $z \in \textbf{B}_0$ implies $r \in [\frac12,2]$. Also by \eqref{inters01} and by $\dz < \frac{\tau^2}{640}$ from \eqref{para12}, we can apply Lemma \ref{geom3} with $\triangle ABC= \triangle x_+x_-x_\times $, $a=\dz$, $b=r$ and $c=\frac{\tau}4$ to deduce that
$$ \text{diam} \, \Gamma \lesssim \frac{\dz}{\tau^2}. $$

Recall $V$ is a $\dz$-separated set in $\textbf{B}_0 \subset \rr^3$. Then
 for any $z,z' \in V \cap \Gamma$,
 $$  B(z,\frac{\dz}{3}) \cap B(z',\frac{\dz}{3}) = \emptyset,$$
 which, together with diam$(V \cap \Gamma) \lesssim \dz\tau^{-2}$, implies
 $$ \# (V \cap \Gamma) \dz^3 \sim \# (V \cap \Gamma) |B(z,\frac{\dz}{3})| = \lf |\bigcup_{z \in V \cap \Gamma }B(z,\frac{\dz}{3}) \lr|  \lesssim [\text{diam} \, (V \cap \Gamma)]^3  \lesssim \dz^3\tau^{-6}.  $$
 Hence $\# (V \cap \Gamma) \lesssim \tau^{-6}$.
 We can deduce that there are at most only $\lesssim \tau^{-6} $ many $z \in V$ satisfying \eqref{inters} for fixed $i_+,i_-$ and $i_\times$. As a consequence, we have
\begin{equation}\label{inters2}
  \# \ct \lesssim \# \ci_{k_1} \times \# \ci_{k_1}\times \# \ci_{k_1}\times \tau^{-6} \lesssim (\# \ci_{k_1})^3 \tau^{-6} \lesssim_{s'} (\# \ci_{k_1})^3(\log \frac{1}{\dz})^{12/s'},
\end{equation}
which completes the proof of \textit{Step 3}.

\bigskip
\textit{Step 4.}
We estimate $\#\ct$ from below. To this end,
recall $F_{k_1}=\cup_{i \in \ci_{k_1}} B(x_i,r_i)$. Hence for any $z \in V$, we have
$$h^+_{z} \cap F_{k_1} \subset \bigcup_{i \in \ci_{k_1}} B(x_i,r_i), \quad  h^-_{z} \cap F_{k_1} \subset \bigcup_{i \in \ci_{k_1}} B(x_i,r_i), \quad  h^\times_{z} \cap F_{k_1} \subset \bigcup_{i \in \ci_{k_1}} B(x_i,r_i).$$
For each $z \in V$, define
$$  \ci_{k_1}^+(z) := \{i \in \ci_{k_1} \ | \ h^+_{z} \cap B(x_i,r_i) \ne \emptyset\}, \quad  \ci_{k_1}^-(z) := \{i \in \ci_{k_1} \ | \ h^-_{z} \cap B(x_i,r_i) \ne \emptyset\}, $$
and
$$  \ci_{k_1}^\times(z) := \{i \in \ci_{k_1} \ | \ h^\times_{z} \cap B(x_i,r_i) \ne \emptyset\}. $$
 With the help of \eqref{tri10}, we have
$$ (\log\frac1{\dz})^{-2}  \lesssim_{s'} \ch_\fz^{s'}(h^+_{z} \cap F_{k_1}) \le \sum_{i \in \ci_{k_1}^+(z)} (\mbox{diam} B(x_i,r_i))^{s'} \sim  \sum_{i \in \ci_{k_1}^+(z)} \dz^{s'} \le \#\ci_{k_1}^+(z) \dz^{s'} \mbox{ for all $z \in V$}, $$
which implies
\begin{equation}\label{+}
  \#\ci_{k_1}^+(z) \gtrsim_{s'} \frac1{\dz^{s'}}(\log\frac1{\dz})^{-2} \mbox{ for all $z \in V$}.
\end{equation}
Similarly, we have
\begin{equation}\label{-}
  \#\ci_{k_1}^-(z) \gtrsim_{s'} \frac1{\dz^{s'}}(\log\frac1{\dz})^{-2} \mbox{ and } \#\ci_{k_1}^\times(z) \gtrsim_{s'} \frac1{\dz^{s'}}(\log\frac1{\dz})^{-2} \mbox{ for all $z \in V$}.
\end{equation}
On the other hand,
recalling the definition of $\ct$ in thew end of \textit{Step 2}, we know
$$ \ct = \bigcup_{z \in V} \ci_{k_1}^+(z)\times \ci_{k_1}^-(z)\times \ci_{k_1}^\times(z) \times\{z\}. $$
Employing the lower bounds in \eqref{+} and \eqref{-}, we arrive at
$$ \# \ct \ge  \min_{z \in V}\lf\{\#\ci_{k_1}^+(z) \lr \}\times \min_{z \in V} \lf \{\#\ci_{k_1}^-(z) \lr \} \times \min_{z \in V} \lf \{\#\ci_{k_1}^\times(z) \lr \} \times \# V \gtrsim_{s'} \frac1{\dz^{3s'}}(\log\frac1{\dz})^{-6} \# V.$$
Combining \eqref{lbdd21} and \eqref{bd11} we conclude
$$  \# \ct \gtrsim_{s'} \frac1{\dz^{3s'+t'}}(\log\frac1{\dz})^{-8}. $$
Recalling \eqref{inters2} we obtain
$$ \# \ci_{k_1} \gtrsim_{s'} (\frac1{\dz^{3s'+t'}}(\log\frac1{\dz})^{-8})^{1/3} (\log \frac{1}{\dz})^{-12/s'} = \frac1{\dz^{s'+t'/3}}(\log\frac1{\dz})^{-(\frac83+\frac{12}{s'})}. $$

We deduce that
\begin{align*}
  \sum_{i\in \ci} r_i^{s'+t'/3-\ez} & \ge \sum_{i\in \ci_{k_1}}  r_i^{s'+t'/3-\ez} \\
   & \gtrsim_{s'}  2^{-k_1(s'+t'/3-\ez)}\frac1{\dz^{s'+t'/3}}(\log\frac1{\dz})^{-(\frac83+\frac{12}{s'})} \\
   & \gtrsim_{s'} \dz^{-\ez}(\log\frac1{\dz})^{-(\frac83+\frac{12}{s'})} \\
   & > 1.
\end{align*}
where in the third inequality we recall $\dz=2^{-k_1}$ and in the last inequality we recall \eqref{para11}. This enables us to deduce
$$ \dim_\ch(F) \ge s'+\frac{t'}{3}- \ez$$
for any $0<s'<s$, $0<t'<t$ and $\ez>0$. Therefore,
$$ \dim_\ch(F) \ge s+\frac{t}{3}.$$
    We conclude the proof.
\end{proof}

\section{Proof of Theorem 1.3}

To show Theorem \ref{main3}, we define the multiplicity function $m^\mu_\dz(w): \rr^2 \to [0,1]$ with respect to a finite measure $\mu$ on $\rr^3$:
\begin{equation}\label{defmu0}
  m^\mu_\dz(w):= \mu(\{z \in \rr^3 \ | \ w \in S^\dz(z)\}).
\end{equation}

We recall \cite[Lemma 5.1]{kov17}, which is a variant of Schlag's weak type inequaltiy \cite[Lemma 8]{sch03} and the main lemma in \cite{w97} by Wolff:
\begin{lem} \label{low}
  Fix $t \in(0,1]$, $\dz > 0, \eta > 0, \textbf{C} \ge 1$, and $A \ge C_{\eta,\bfc,t}\cdot\dz^{-\eta}$, where $C_{\eta,\bfc,t} \ge 1$ is a
large constant depending only on $\eta, \bfc$ and $t$. Let $\mu$ be a probability measure on $\rr^3$ satisfying the
Frostman condition $\mu(B(z,r)) \le  \bfc r^t$ for all $z \in \rr^3$ and $r>0$, and with $D := \mbox{\rm spt } \mu \subset \textbf{B}_0$ where $\textbf{B}_{0}$ is defined in \eqref{b0}.
Then, for $\lz \in (0, 1]$, there is a set $G(A, \dz, \lz) \subset D$  with
$$  \mu(D\setminus G(A, \dz, \lz)) < A^{-t/3}$$
such that the following holds for all $z \in G(A, \dz, \lz)$:
$$ |S^\dz(z) \cap \{ w \ | \ m^\mu_\dz(w) \ge A^t\lz^{-2t}\dz^t\}| \le \lz |S^\dz(z)|.  $$
\end{lem}

\begin{rem}
  \rm We remark that the assumptions on $\mu$ in Lemma \ref{low} can be slightly relaxed, which means we can apply Lemma \ref{low} for measures $\mu$ satisfying that
  \begin{enumerate}
    \item [$(i)$] $\mu$ is a finite measure with total mass smaller or equal to $1$ supported on $\textbf{B}_{0}$;
    \item [$(ii)$] $\mu$ enjoys Frostman condition
    $$\mu(B(z,r)) \le  \bfc r^t \mbox{  for all $z \in \rr^3$ and $r>\dz$}.$$
  \end{enumerate}

Indeed, in the proof of \cite[Lemma 5.1]{kov17}, the fact that the total measure $\mu(D)= 1$ was only used at the beginning to reduce the proof to the case that $\dz$ is small. See the first paragraph of the proof therein. Moreover, the Frostman condition was only applied to balls in $\rr^3$ with radius $\dz < r \in [C\dz, 1]$ where $C \ge 1$ in their proof. See the inequality above (5.4), the definition of $B$ below (5.22) and  inequality (5.24) therein. Hence we can reduce the assumptions in Lemma \ref{low} to $(i)$ and $(ii)$ above for the measure $\mu$.
\end{rem}

\begin{proof}[Proof of Theorem \ref{main3}.]
Let $F$ be a circular $(s,t)$-Furstenberg set with parameter set $K \subset \textbf{B}_0$.
It suffices to show, for any $\ez>0$, $\frac12<s'<s$ and $0<t'<t$,
$$ \dim_\ch(F) \ge (2s'-1)t'+s'- \ez.$$
Hence in the following, we fix $\ez, s',t'$.

Let $\az>0$ and $K_1$ be as in \eqref{lbdd11}. Now we clarify the choices of parameters appeared in the ensuing proof and we remind that all parameters are unrelated to those in the proof of Theorem \ref{main2}. First, we choose
\begin{equation}\label{para0}
  \eta = \min\{\ez /2t',  (2s'-1)/2 \}.
\end{equation}
Then there exists $\dz_0=\dz_0(\ez, s',t')>0$ such that for any $0<\dz<\dz_0$, we have
\begin{equation}\label{para6}
  \dz^{\ez-t'\eta}(\log\frac1\dz)^{6+4t'} \le \dz^{\frac{\ez}2}(\log\frac1{\dz})^{6+4t'} < 1,
\end{equation}
\begin{equation}\label{para1}
  \dz^{\frac{\eta t'}3}(\log\frac1{\dz})^{2} < \frac14,
\end{equation}
and
\begin{equation}\label{para5}
  C_{\eta,\bfc,t',s'} = (C_{\eta,\bfc,t'})^{t'}(2c_0 4^{s'})^{2t'},
\end{equation}
where $\bfc$ and $C_{\eta,\bfc,t'} \ge 1$ are the constants appeared in Lemma \ref{low} and $c_0$ is as in Remark \ref{rem1}(ii), i.e. $|S^{\dz}(x,r)| \le c_0 \dz$ for all $(x,r) \in \textbf{B}_0$.

Let $k_0$ be the smallest integer larger than $(\log\frac1{\dz_0})$ also satisfying
\begin{equation}\label{para3}
  \az>\sum_{k=k_0}^{\fz} \frac{1}{k^2}.
\end{equation}

Now, we outline the main steps of the proof. We start with an arbitrary cover  $\cu = \{ B(x_i,r_i)\}_{i \in\ci}$ of $F$
by balls of radius less than $2^{-k_0}$. In the sequel, we will derive a lower bound
$$ \sum_{i \in \ci} r_i^\sz \gtrsim_{\ez,t',s'} 1$$
with $\sz=(2t'+1)s'-t'- \ez$ independent of the choice of the particular cover. This will imply
$$ \ch^\sz(F) > 0.$$

To this end, we divide the proof into 3 steps.
Let
$$ \ci_k:= \{i \in \ci \ | \  2^{-(k+1)}<r_i \le 2^{-k} \}, \ F_k:= \{\bigcup B(x_i,r_i) \ | \ i \in \ci_k \}.$$

First, in  \textit{Step 1}, we will deduce that there exists $k_1 \ge k_0$ and a $(\dz,t')$-set $V \subset K$ with $\dz=2^{-k_1}$ such that
\begin{equation}\label{bd1}
  \frac1{k_1^2} \cdot \dz^{-t'} \lesssim \# V \lesssim \dz^{-t'} ,
\end{equation}
and for every circle $z=(x,r) \in V$, we have
\begin{equation}\label{lbdd3}
  \ch^{s'}_\fz(S(z)\cap F_{k_1})> k_1^{-2}.
\end{equation}

Next, in  \textit{Step 2}, we associate a finite measure $\mu$ supported on $V$ using \eqref{defmu}. Then we apply Lemma \ref{low} to obtain that there exists $G \subset V$ and $S_2^\dz(z)$ contained in the $\dz$-neighbourhood of $S(z)\cap F_{k_1}$, such that for every $z \in G$ and $w \in S_2^\dz(z)$,
\begin{equation}\label{share0}
   \# \{z' \in G \ | \ w \in S_2^\dz(z')\} \lesssim_{t'} C_{\eta,\bfc,t',s'} \dz^{t'(2s'-2-\eta)} (\log\frac1{\dz})^{4t'+2}.
\end{equation}

Finally, in  \textit{Step 3}, we will provide a lower bound of the cardinality $\# \ci_{k_1}$ by combining the upper bound in \textit{Step 2} as well as the lower bounds on the cardinality $\# G$ and the Lebesgue measure $|S_2^\dz(z)|$. Explicitly, we have
$$  \#\ci_{k_1} \gtrsim_{\ez,t',s'}   \frac1{\dz^{(2t'+1)s'-t'(1+\eta)}}\frac1{(\log\frac1\dz)^{6+4t'}}.  $$
This will enable us to conclude the proof.

\bigskip
\textit{Step 1.} Employing the same arguments as in \textit{Step 1} in the proof of Theorem \ref{main2}, we can deduce the existence of $k_1$ and $V \subset K_2 \subset K_1$ satisfying \eqref{lbdd3} and the first inequality in \eqref{bd1}. The second inequality in \eqref{bd1} is derived from Remark \ref{remsim}. Here, we omit the details.
%
%
%
%

\bigskip
\textit{Step 2.}
  Define $\mu_{V}$ as in \eqref{defmu} applied to $P=V$.  Then we know $\mu_{V}$ is a probability measure satisfying the Frostman condition $$\mu_{V}(B(z,r)) \le C\ch^{t'}_\fz(K_2)^{-1} r^{t'} < C k_1^{2} r^{t'} = C(\log\frac1{\dz})^2 r^{t'}$$
  for all $z \in \rr^3$ and $r>\dz$. Hence by setting $$\mu:=\frac{\mu_{V}}{(\log\frac1{\dz})^2},$$
we know that $\mu$ has total measure $(\log\frac1{\dz})^{-2}<1$, spt$\mu=V \subset  \textbf{B}_0$ and
$$\mu(B(z,r)) \le Cr^{t'} =:  \bfc r^{t'}$$
  for all $z \in \rr^3$ and $r>\dz$.

Let $m_\mu^\dz$ be the corresponding multiplicity function with respect to $\mu$ defined as in \eqref{defmu0}.

Applying Lemma \ref{low} with $t=t'$, $\dz=2^{-k_1}$, $\eta$ as in \eqref{para0}, $\mu=(\log\frac1{\dz})^{-2}\mu_{V}$, $D=V$ and
\begin{equation}\label{lambda}
  \lz=(2c_0 4^{s'}k_1^2)^{-1} \dz^{1-s'},
\end{equation}
we obtain that for $A= C_{\eta,\bfc,t'}\cdot\dz^{-\eta}$, there is a set $G=G(k_1,s',t',\ez) \subset V$  with
\begin{equation}\label{bd2}
  \mu(V \setminus G) < A^{-t'/3}
\end{equation}
such that the following holds for all $z \in G$:
\begin{equation}\label{low2}
   |S^\dz(z) \cap \{ w \ | \ m^\mu_\dz(w) \ge A^{t'}\lz^{-2t'} \dz^{t'} \}| \le \lz |S^\dz(z)|.
\end{equation}
Because $|S^\dz(z)| \le c_0 \dz$ for all $z \in \textbf{B}_0$, \eqref{low2} becomes
 \begin{equation}\label{low3}
   |S^\dz(z) \cap \{ w \ | \ m^\mu_\dz(w) \ge A^{t'}\lz^{-2t'} \dz^{t'} \}| \le c_0\lz \dz.
\end{equation}
Moreover, recalling that $\dz=2^{-k_1}$ and $k_1 \ge k_0$, we know that $0<\dz<\dz_0$. Hence by $C_{\eta,\bfc,t'} \ge 1$, \eqref{para1} and the choice of $\eta$ in \eqref{para0} we deduce
$$  A^{-\frac{t'}3} \le \dz^{\frac{\eta t'}3} \le \frac14 \frac{1}{(\log\frac1{\dz})^2}=\frac14\mu(V).$$
Hence \eqref{bd2} becomes
\begin{equation}\label{bd3}
  \mu(V \setminus G ) < \frac14\mu(V).
\end{equation}

For $z\in G$, let $S_1(z):=S(z)\cap F_{k_1}$ and $S^\dz_1(z)$ be the $\dz$-neighbourhood of $S_1(z)$. Our next goal is to substitute the right hand side term $\lz|S^\dz(z)|$ in \eqref{low2} by the term $\frac12 |S^\dz_1(z)|$ with the help of the proper choice of $\lz$ as in \eqref{lambda}. This means, in the sense of $2$-dimensional Lebesgue measure, more than half of the points in $S^\dz_1(z)$ have low multiplicity. To this end,
we claim that
\begin{equation}\label{sim2}
  |S^\dz_1(z)| \ge  \frac{1}{4^{s'}k_1^2} \dz^{2-s'}.
\end{equation}
To see \eqref{sim2}, let $P(z)$ be a maximal $2\dz$-separated set in $S_1(z)$. Then $\cup_{p \in P(z)} B(p,2\dz)$ forms a cover of $S_1(z)$. Hence
$$ \ch_\fz^{s'}(S_1(z)) \le \#P(z)(4\dz)^{s'}.  $$
which, combined with \eqref{lbdd3}, implies
$$\# P(z) \ge \ch_\fz^{s'}(S_1(z))\frac1{(4\dz)^{s'}} \ge \frac{1}{4^{s'}k_1^2}\frac1{\dz^{s'}}. $$
On the other hand, we have $\cup_{p \in P(z)} B(p,\dz) \subset S^\dz_1(z)$. Hence by $\{B(p,\dz)\}_{p \in P(z)}$ being mutually disjoint, we deduce
$$ |S^\dz_1(z)| \ge |\cup_{p \in P(z)} B(p,\dz)| = \# P \dz^2 \pi \ge \frac{\pi}{4^{s'}k_1^2}\dz^{2-s'}>\frac{1}{4^{s'}k_1^2} \dz^{2-s'},  $$
which gives \eqref{sim2}.

Noticing that $|S^\dz(z)| \le c_0 \dz$, $S_1(z) \subset S(z)$ and combining \eqref{low3} as well as \eqref{sim2}, we arrive at
\begin{equation}\label{con1}
  |S_1^\dz(z) \cap \{ w \ | \ m^\mu_\dz(w) \ge A^{t'}\lz^{-2t'} \dz^{t'}  \}| \le \lz c_0 \dz \le \lz c_0 4^{s'}k_1^2\dz^{s'-1} |S_1^\dz(z)|.
\end{equation}
Now recall $A =  C_{\eta,\bfc,t'}\cdot\dz^{-\eta}$ and $\lz=(2c_0 4^{s'}k_1^2)^{-1}\dz^{1-s'}=(2c_0 4^{s'})^{-1}(\log\frac1{\dz})^{-2}\dz^{1-s'}$.
Then \eqref{con1} becomes
\begin{equation*}
  |S_1^\dz(z) \cap \{ w \ | \ m^\mu_\dz(w) \ge C_{\eta,\bfc,t',s'} \dz^{t'(2s'-1-\eta)}(\log\frac1{\dz})^{4t'} \}| \le \frac12 |S_1^\dz(z)|
\end{equation*}
where we recall $C_{\eta,\bfc,t',s'}$ defined in \eqref{para5}.

For each $z \in G$, define the low-multiplicity set
$$S^\dz_2(z):= \{ w \in S_1^\dz(z)  \ | \  m^\mu_\dz(w) < C_{\eta,\bfc,t',s'} \dz^{t'(2s'-1-\eta)} (\log\frac1{\dz})^{4t'} \}.$$
Then we have
\begin{equation}\label{con2}
 |S^\dz_2(z)| \ge \frac12 |S_1^\dz(z)|.
\end{equation}
See Figure \ref{f7} for an illustration of $S_1(z)$, $S_1^\dz(z)$ and $S_2^\dz(z)$.
\vspace*{7pt}
\begin{figure}[h]
\centering
\includegraphics[width=15cm]{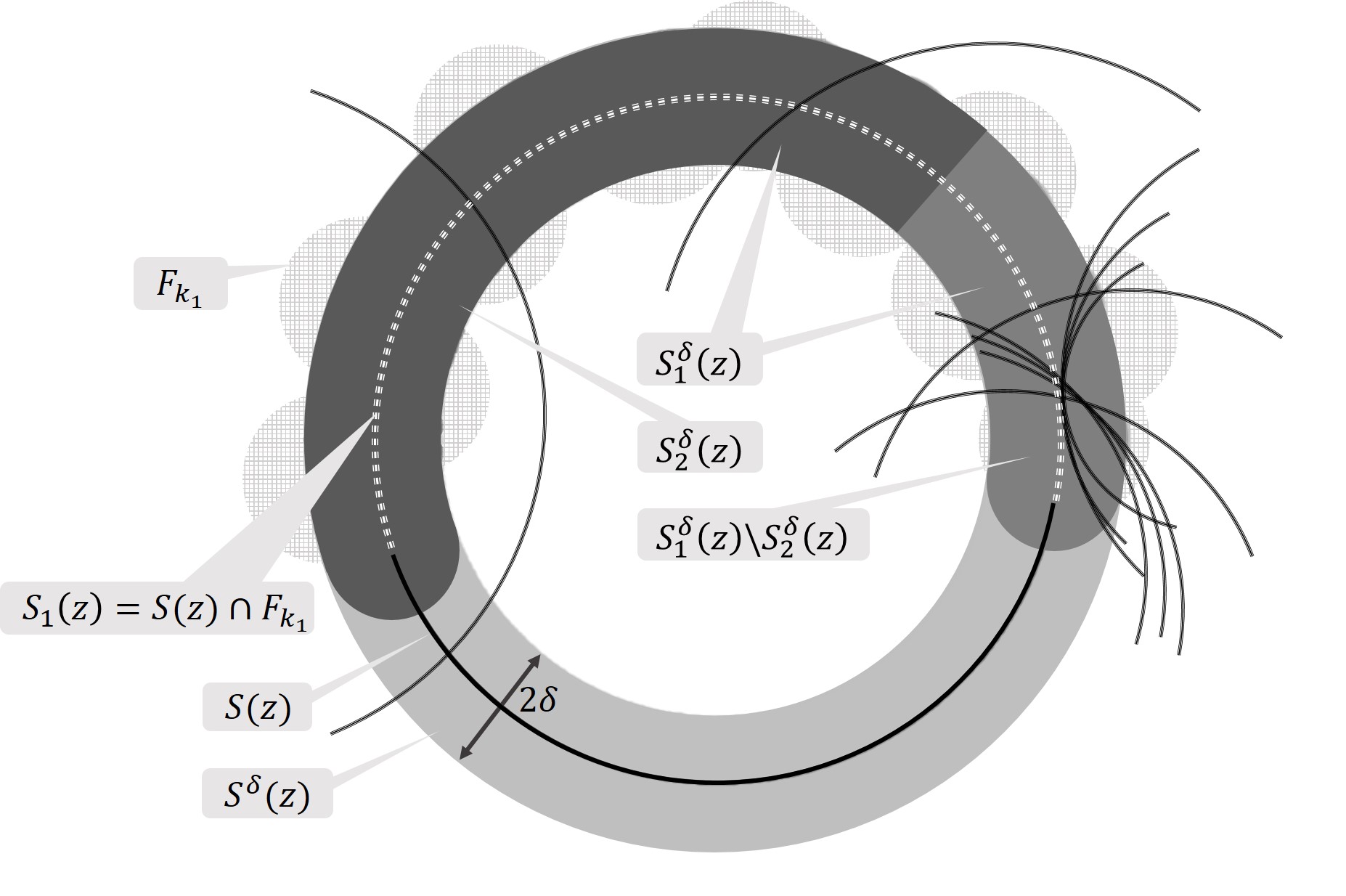}
\caption  {An illustration of $S_1(z)$, $S_1^\dz(z)$ and $S_2^\dz(z)$.}
\label{f7}
\end{figure}

Notice that $m^\mu_\dz(w) < C_{\eta,\bfc,t',s'} \dz^{t'(2s'-1-\eta)} (\log\frac1{\dz})^{4t'}$ is equivalent to
    $$ \mu(\{z' \in \rr^3 \ | \ w \in S^\dz(z')\}) < C_{\eta,\bfc,t',s'} \dz^{t'(2s'-1-\eta)} (\log\frac1{\dz})^{4t'} , $$
    which, combined with \eqref{bd1}, indicates that for $w
\in S^\dz_2(z)$, it holds
\begin{equation*}
  \# \{z' \in V \ | \ w \in S^\dz(z')\} \le \# V \cdot C_{\eta,\bfc,t',s'} \dz^{t'(2s'-1-\eta)} (\log\frac1{\dz})^{4t'+2} \lesssim_{t'} C_{\eta,\bfc,t',s'} \dz^{t'(2s'-2-\eta)} (\log\frac1{\dz})^{4t'+2}.
\end{equation*}
Furthermore, by the inclusions $G \subset V$ and $S_2^\dz(z) \subset S^\dz(z)$,
we conclude \eqref{share0}, which finishes \textit{Step 2.}

\bigskip
\textit{Step 3.} We will lower bound $\# \ci_{k_1}$ in the following. First notice that if $\{S^\dz(z)\}_{z \in G}$ were mutually disjoint, we could lower bound $\# \ci_{k_1}$ by summing up the number of balls $B_i \ (i \in \ci_{k_1})$ needed to cover each $S_2^\dz(z)$ since no ball could simultaneously intersect two of these sets. However, in general, $\{S^\dz(z) \}_{z \in G}$ may not be mutually disjoint, which needs a bit more efforts to get the lower bound of $\# \ci_{k_1}$.

Let
$$  \wz F _{k_1} : =  \bigcup_{i \in \ci_{k_1} } B(x_i, 4r_i)  .$$
We deduce that
\begin{equation}\label{ubdd1}
 \bigcup_{z \in G } S^\dz_2(z) \subset \wz F _{k_1} .
\end{equation}
Indeed, for any $w \in S^\dz_2(z)$, there exists $w' \in S(z) \cap F _{k_1}$ such that
$$  \|w-w'\| < \dz.$$
On the other hand, we know that $w' \in B(x_i,r_i)$ for some $x_i \in \ci_{k_1}$ and $r_i>2^{-(k_1+1)}=\dz/2$, which implies
$$  \|w'-x_i\| < r_i $$
and hence
$$  \|w-x_i\| < \dz+r_i < 3r_i. $$
In addition, by \eqref{bd1} and \eqref{bd3}, we can infer that
\begin{equation}\label{share3}
 \# G \gtrsim \# V \gtrsim \frac1{\dz^{t'}}\frac1{(\log\frac1\dz)^2}.
\end{equation}
Moreover by recalling \eqref{share0} we obtain that for every $w \in \bigcup_{z \in G } S^\dz_2(z) $,
$$  \mathcal N (w):= \# \{z' \in G \ | \ w \in S_2^\dz(z')\} \lesssim_{t'} C_{\eta,\bfc,t',s'} \dz^{t'(2s'-2-\eta)} (\log\frac1{\dz})^{4t'+2}$$
and hence combining \eqref{share3}, we can estimate
\begin{align}\label{ubdd2}
\lf | \bigcup_{z \in G } S^\dz_2(z) \lr| &=  \sum_{z \in G} \int \chi_{S^\dz_2(z)}(w) \frac1{\mathcal N(w) } \, dw  \notag \\
& \gtrsim_{t'} (C_{\eta,\bfc,t',s'} \dz^{t'(2s'-2-\eta)} (\log\frac1{\dz})^{4t'+2})^{-1}   \sum_{z \in G }   \lf |  S^\dz_2(z) \lr| \notag \\
& \gtrsim_{t'} (C_{\eta,\bfc,t',s'} )^{-1} \frac1{\dz^{t'(2s'-2-\eta)}}  \frac1{(\log\frac1\dz)^{4t'+2}}  \frac1{\dz^{t'}}\frac1{(\log\frac1\dz)^2} \min_{z \in G } \{|  S^\dz_2(z) |\} \notag \\
& \gtrsim _{\eta,t',s'} \frac1{\dz^{t'(2s'-2-\eta)}}  \frac1{(\log\frac1\dz)^{4t'+2}}  \frac1{\dz^{t'}}\frac1{(\log\frac1\dz)^2}  \dz^{2-s'} \frac1{(\log\frac1\dz)^2}
\end{align}
where in the last inequality we employ \eqref{sim2} and \eqref{con2}.
Therefore, combining \eqref{ubdd1} and \eqref{ubdd2} we arrive at
$$   \#\ci_{k_1} \dz^2 \gtrsim   \lf |   \wz F_{k_1}\lr|  \ge  \lf| \bigcup_{z \in G } S^\dz_2(z) \lr|  \gtrsim _{\eta,t',s'}  \frac1{\dz^{t'(2s'-2-\eta)}}  \frac1{(\log\frac1\dz)^{4t'+2}}  \frac1{\dz^{t'}}\frac1{(\log\frac1\dz)^2}  \dz^{2-s'} \frac1{(\log\frac1\dz)^2}, $$
which implies
$$  \#\ci_{k_1} \gtrsim_{\eta,t',s'}   \frac1{\dz^{(2t'+1)s'-t'(1+\eta)}}\frac1{(\log\frac1\dz)^{6+4t'}}.  $$

Since $\ci_{k_1} \subset \ci$, we deduce that
\begin{align*}
  \sum_{i\in \ci} r_i^{(2t'+1)s'-t'-\ez} & \ge\sum_{i\in \ci_{k_1}}  r_i^{(2t'+1)s'-t'-\ez} \\
   & \gtrsim_{\eta(\ez,t',s'),t',s'}  2^{-k_1((2t'+1)s'-t'-\ez)}\frac1{\dz^{(2t'+1)s'-t'(1+\eta)}}\frac1{(\log\frac1\dz)^{6+4t'}}  \\
   & \gtrsim_{\eta(\ez,t',s'),t',s'}  \dz^{t'\eta-\ez}\frac1{(\log\frac1\dz)^{6+4t'}}  \\
   & \gtrsim_{\ez,t',s'}  \dz^{-\ez/2}\frac1{(\log\frac1\dz)^{6+4t'}} \\
   & >1.
\end{align*}
where in the third inequality we recall $\dz=2^{-k_1}$ and in the fourth as well as the last inequality we recall \eqref{para6}. This enables us to deduce
$$ \dim_\ch(F) \ge (2t'+1)s'-t'- \ez$$
for any $\frac12<s'<s$, $0<t'<t$ and $\ez>0$. Therefore,
$$ \dim_\ch(F) \ge (2t+1)s-t = (2s-1)t+s.$$
    We conclude the proof.
\end{proof}

\section{Proof of Lemma 2.5}

This section is devoted to the proof of Lemma \ref{geom3}.
For the readers' convenience, we restate Lemma \ref{geom3} in the following.

\begin{lem}
  Let $A,B,C \in \rr^2$ such that $\min\{\|A-B\|, \|A-C\|, \|B-C\|\} \ge 2c$ with $c<1$. For $a>0$ such that $a < \frac1{20}c^2$, define
  \begin{equation*}
W:= \left \{ \begin{array}{ll}
  \quad & \quad b-a \le \|x-A\| \le b+a,  \\
 (x,b) \in  \rr^2 \times [\frac12,2]:  &  \quad b-a \le \|x-B\| \le b+a, \\
   \quad  & \quad b-a \le \|x-C\| \le b+a 
\end{array}
\right \}.
\end{equation*}
Then
\begin{equation*}
  \text{\rm diam} \, W \lesssim \frac{a}{c^2}.
\end{equation*}
\end{lem}

We briefly explain the approach. We will decompose $W$ as
$$  W = \bigcup_{b \in I \subset [1/2,2]} W(b) \times \{b\}.$$
Then for each fixed $b$,
\begin{equation*}
W(b):= \left \{ \begin{array}{ll}
  \quad & \quad b-a \le \|x-A\| \le b+a,  \\
 x \in  \rr^2 :  &  \quad b-a \le \|x-B\| \le b+a, \\
   \quad  & \quad b-a \le \|x-C\| \le b+a 
\end{array}
\right \}
=S^a(A,b) \cap S^a(B,b) \cap S^a(C,b)
\end{equation*}
is a subset in $\rr^2$ formed by the intersection of three annuli.
We will show that $W(b) \ne \emptyset$ only for $b$ ranging in a set $I$ with diameter $\lesssim \frac{a}{c^2}$. Moreover, if $W(b) \ne \emptyset$, then $A,B,C$ form a non-degenerate $\triangle ABC$ with circumcenter $M$ and $W(b)$ is contained in a rhombus centered at $M$ with diameter $\lesssim \frac{a}{c^2}$.
This will imply
$$\text{diam} \, W \lesssim \frac{a}{c^2}.$$
The above justification is contained in next two auxiliary lemmas.
In what follows, given $A, B \in \rr^2$ and $0<a<\frac{c^2}{20}$, we denote by $\mathcal{R}_{AB}^{a,c}$ the rectangle centered at the middle point of $AB$ whose short sides have length $ \frac{9a}{c}$ and long sides have length $6$ parallel to the bisector of $AB$.

\begin{lem}\label{geom1}
    Let $A,B \in \rr^2$ and $b \in [\frac12,2]$. If $c<\min\{1, \frac{\|A-B\|}{2}\}$ and $0<a<\frac{c^2}{20}<1$, then
    $$  S^a(A,b) \cap S^a(B,b)\subset \mathcal{R}^{a,c}_{AB} . $$
\end{lem}

\begin{proof}
   Let $\|A-B\|=2u$. Without loss of generality, we assume $A=(-u,0)$ and $B=(u,0)$. It is easy to see that
  \begin{align*}
    & S^a(A,b) \cap S^a(B,b) \\
     & \quad = \{ x  \in \rr^2 \ | \ b-a \le \|x-A\| \le b+a, \ b-a \le \|x-B\| \le b+a \}  \\
     & \quad \subset U:= \{ x =(x_1,x_2) \in \rr^2 \ | \ \max \{\|x-A\|,\|x-B\| \}\le3, \ -2a \le \|x-A\| - \|x-B\| \le 2a \}.
  \end{align*}

 Since $u = \frac{\|A-B\|}2 > c> a$, from planar geometry we know that the set
$$\{x \in \rr^2 \ | \   \|x-A\| - \|x-B\|  = \pm 2a \}$$
consisting of points, whose absolute difference of distances to the two fixed points $A$ and $B$ is the constant $2a$,
is a hyperbola in $\rr^2$ determined by the equation
$$y(x) = y(x_1,x_2) =1$$
 where
   $y : \rr^2 \to \rr$ is defined  by
  $$ y(x)=y(x_1,x_2) \mapsto \frac{x_1^2}{a^2} - \frac{x_2^2}{u^2-a^2}.  $$
  Then we observe that
  $$\{x \in \rr^2 \ | \  -2a \le \|x-A\| - \|x-B\|  \le 2a \} = \{x  \in \rr^2 \ | \ y(x_1,x_2) \le 1 \} $$
   and hence
  \begin{align*}
     U
     & =[B((-u,0),3) \cap  B((u,0),3)] \cap \{x  \in \rr^2 \ | \ y(x_1,x_2) \le 1 \},
  \end{align*}
  which implies
  $$ U \subset \{ x  \in \rr^2 \ | \ |x_2|\le3, \  y(x_1,x_2) \le1 \}.  $$
  Figure \ref{f0201} shows the case that $u=2$ and $a=0.75$.
  \vspace*{7pt}
\begin{figure}[h]
\hspace{-0.7cm}
\centering
\includegraphics[width=13cm]{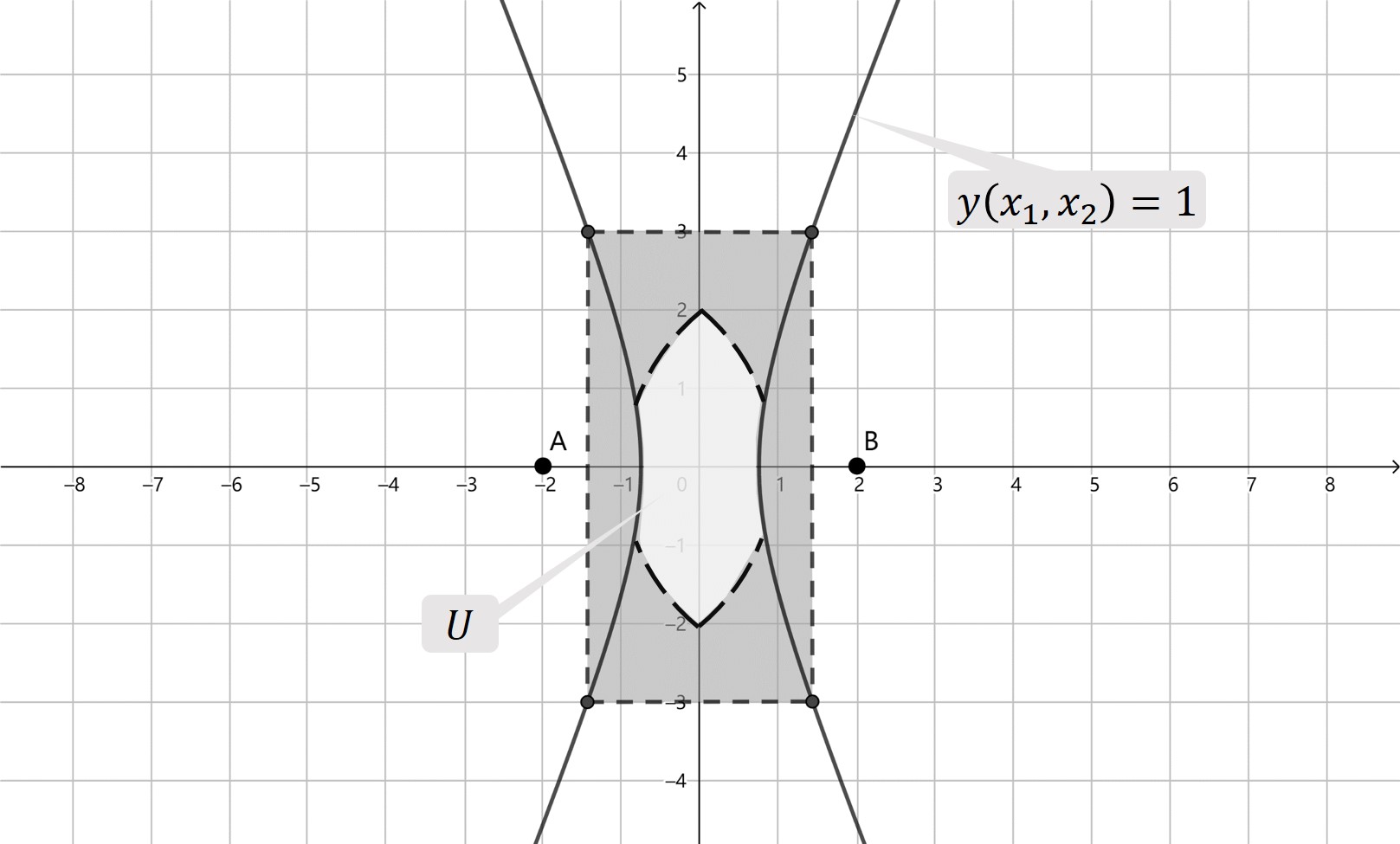}
\caption  {The case $u=2$ and $a=0.75$.}
\label{f0201}
\end{figure}

  Letting $|x_2|=3$ in the equation $y(x_1,x_2)=1$, we have $|x_1|=a\sqrt{1+\frac{9}{u^2-a^2}}$.
  Since $20a < c^2<1$ and $u>c$, it holds
\begin{equation}\label{suff12}
 a\sqrt{1+\frac{9}{u^2-a^2}} <a\sqrt{1+\frac{9}{c^2-a^2}} < a\sqrt{1+\frac{9}{\frac89 c^2}}  <  a \sqrt{\frac{81}{4 c^2}} = \frac{9}{2} \frac{a}c <3 .
\end{equation}
  This implies that the rectangle with four vertices $(\pm \frac{9}{2} \frac{a}c, \pm 3)$ has short side length $ \frac{9 a}c$ and long side length $6$. By recalling the definition of $\mathcal{R}^{a,c}_{AB}$, we have
  $$  S^a(A,b) \cap S^a(B,b) \subset U \subset \mathcal{R}^{a,c}_{AB}= \{ x  \in \rr^2 \ | \  |x_1|\le \frac{9}{2} \frac{a}c , \  |x_2|\le3 \},$$
  which concludes the proof.
\end{proof}

\begin{lem}\label{geom2}
    Let $A,B,C \in \rr^2$ such that $\min\{\|A-B\|, \|A-C\|, \|B-C\|,2\} \ge 2c$. Let $b \in [\frac12,2]$. Then for $a>0$ such that $a < \frac1{20}c^2 <b$, define
    \begin{equation}\label{geo1}
W(b):= \left \{ \begin{array}{ll}
  \quad & \quad b-a \le \|x-A\| \le b+a,  \\
 x \in  \rr^2 :  &  \quad b-a \le \|x-B\| \le b+a, \\
   \quad  & \quad b-a \le \|x-C\| \le b+a 
\end{array}
\right \}
\end{equation}
    If the triangle $\triangle ABC$ is degenerate, then
    \begin{equation}\label{g30}
      W(b) = \emptyset  \text{  for all $b \in [\frac12,2]$.}
    \end{equation}

    \noindent
    If $\triangle ABC$ is non-degenerate,
    let $M$ be the circumcenter of $\triangle ABC$ and
    $$ h:= \|M-A\|= \|M-B\|=\|M-C\|.$$
    Then,
    we have
    \begin{equation}\label{geo20}
       W(b) \subset B(M,K\frac{a}{c^2}), \text{ for all $b \in [\frac12,2]$} .
    \end{equation}

\noindent
    In addition,
    if $W(b) \ne \emptyset$, then
    \begin{equation}\label{geo3}
       b \in [h-K\frac{ a}{c^2}, h+ K\frac{ a}{c^2} \, ] \cap [\frac12,2] .
    \end{equation}
    Here in \eqref{geo20} and \eqref{geo3}, $K$ is an absolute constant.
\end{lem}

\begin{proof}
   Without loss of generality, we assume the side $BC$ of $\triangle ABC$ has maximal length. Then $ \angle A:=\angle BAC \ge \pi/3$. Since $W(b)=S^a(A,b) \cap S^a(B,b) \cap S^a(C,b)$,
   from Lemma \ref{geom1} we know
   \begin{equation}\label{g31}
     W(b)  \subset \mathcal{R}_{AB}^{a,c} \cap \mathcal{R}_{AC}^{a,c}.
 \end{equation}
   Below we estimate diam$(\mathcal{R}_{AB}^{a,c} \cap \mathcal{R}_{AC}^{a,c})$ from above.

   Denote by $L_1$ and $L_2$ the bisector of $AB$ and $AC$ respectively. Hence $D:=L_1 \cap AB$ is the middle point of $AB$ and $E:=L_2 \cap AC$ is the middle point of $AC$. See Figure \ref{f0202} for an illustration.
     \vspace*{7pt}
\begin{figure}[ht]
\hspace{-0.7cm}
\centering
\includegraphics[width=10cm]{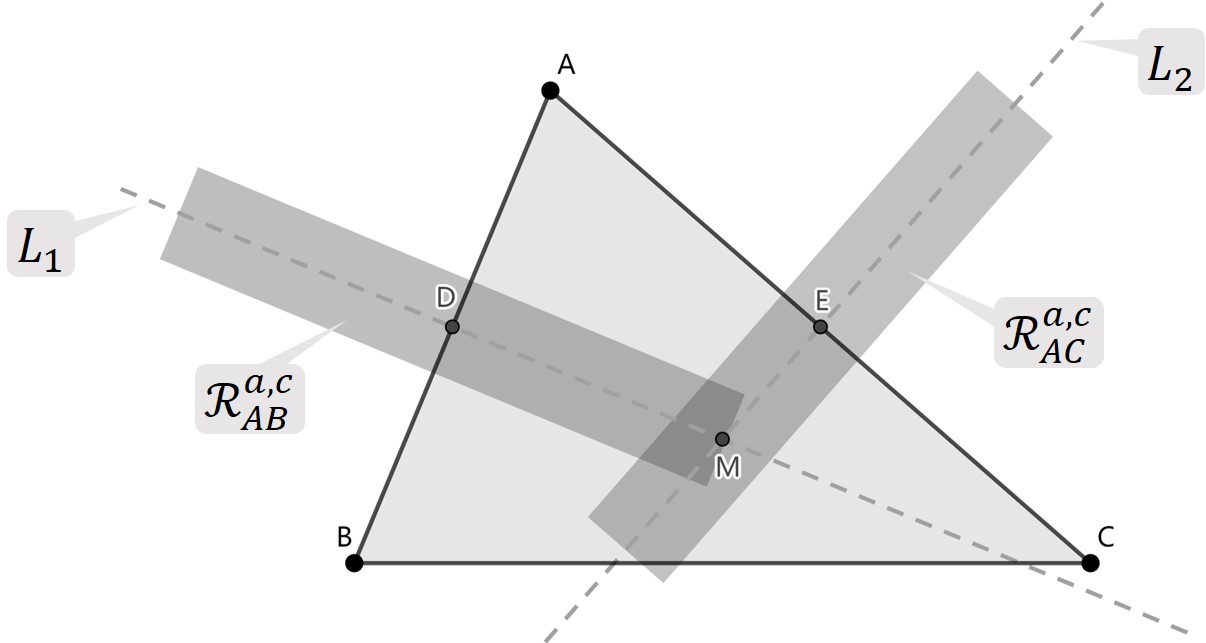}
\caption  {An illustration for $L_1$, $L_2$, $\mathcal{R}_{AB}^{a,c}$ and $\mathcal{R}_{AC}^{a,c}$.}
\label{f0202}
\end{figure}

   Let $d=\frac{9}{2} \frac{a}c$. Since $20a <c^2$, we have
   \begin{equation}\label{suff2}
     d=\frac{9}{2} \frac{a}c < \frac{9}{40}c <\frac14 c .
   \end{equation}

\emph{Case 1:} $\angle A =\pi$. That is, $\triangle ABC$ degenerates.
    By \eqref{suff2}, it is easy to see $\mathcal{R}_{AB}^{a,c} \cap \mathcal{R}_{AC}^{a,c} = \emptyset$, which, with help of \eqref{g31}, implies
    $$ W(b)= \emptyset \text{  for all $b \in [\frac12,2]$.} $$
    That is, \eqref{g30} holds.

     \bigskip
   \emph{Case 2:} $\angle A \in (\pi- \arctan (2c/9),\pi)$. We will show that
   \begin{equation}\label{suff01}
     \mathcal{R}_{AB}^{a,c} \cap \mathcal{R}_{AC}^{a,c} = \emptyset.
   \end{equation}
    Denote $\|A-B\|=2u$ and $\|A-C\|=2v$. Since $M$ is the circumcenter of $\triangle ABC$, it is the intersection of lines $L_1$ and $L_2$. Then the line $L_3$ passing through $A$ and $M$ divides $\rr^2$ into two connected components. Since the center $D$ of $\mathcal{R}_{AB}^{a,c}$ and the center $E$ of $\mathcal{R}_{AC}^{a,c}$ are contained in different connected components above and $d<\frac14 c$ by \eqref{suff2}, a sufficient condition for $\mathcal{R}_{AB}^{a,c} \cap \mathcal{R}_{AC}^{a,c} = \emptyset$ is that
   \begin{equation}\label{suff1}
     \mathcal{R}_{AB}^{a,c} \cap L_3 = \emptyset \text{ and }  \mathcal{R}_{AC}^{a,c} \cap L_3 = \emptyset.
   \end{equation}
   See Figure \ref{f0203} for an illustration.
 \vspace*{7pt}
\begin{figure}[h]
\hspace{-0.7cm}
\centering
\includegraphics[width=10cm]{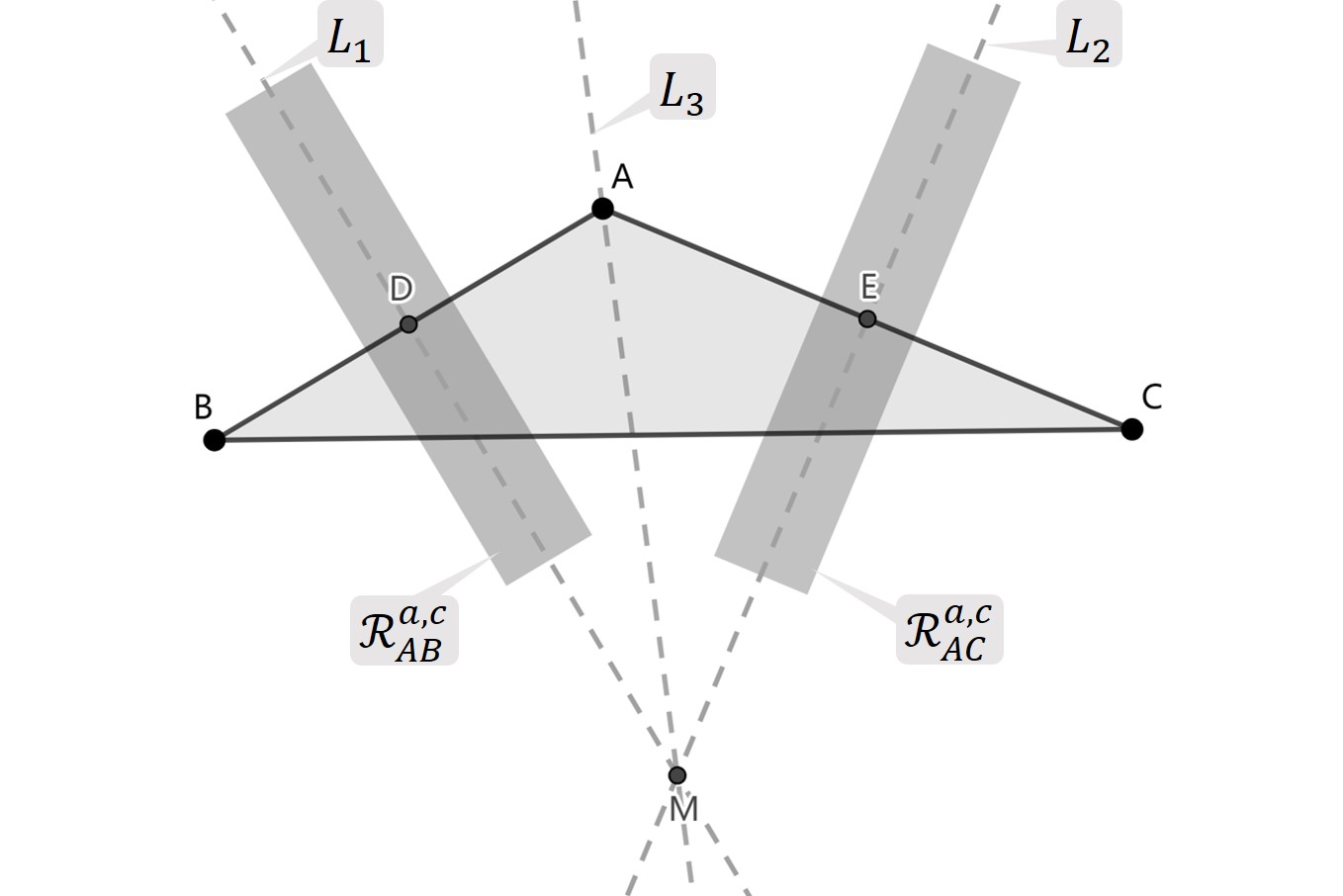}
\caption  {An illustration for \emph{Case 2}.}
\label{f0203}
\end{figure}

    Recall that half of the length of the short sides of $\mathcal{R}_{AB}^{a,c}$ and $\mathcal{R}_{AC}^{a,c}$ is $d=\frac92 \frac{a}c$.
   By assumption $\angle A \in (\pi- \arctan (2c/9),\pi)$, this implies $ \angle DMA + \angle EMA \le \arctan (2c/9)$. Hence
   \begin{equation}\label{suff3}
     \tan \angle DMA < \frac{2c}{9} \le \frac{c-d}{3}  \le \frac{u-d}{3} \text{ and } \tan \angle EMA < \frac{2c}{9} \le \frac{c-d}{3}  \le \frac{v-d}{3}
   \end{equation}
  where in the second inequality we apply $d <\frac{c}3$ from  \eqref{suff2}.
   Now we explain how \eqref{suff3} implies \eqref{suff1}. Let $D'$ be the intersection of the line segment $AD$ and the long side of the triangle $\mathcal{R}_{AB}^{a,c}$.
  Also, let $L_1' : = L_1 + (D'-D)$. That is, line $L_1'$ is the translation of line $L_1$ by the vector $D'-D$ in $\rr^2$. Denote the intersection of $L_1'$ and $L_3$ by $M'$. See Figure \ref{f0207} for an illustration.
   \vspace*{7pt}
\begin{figure}[htb]
\hspace{-0.7cm}
\centering
\includegraphics[width=10cm]{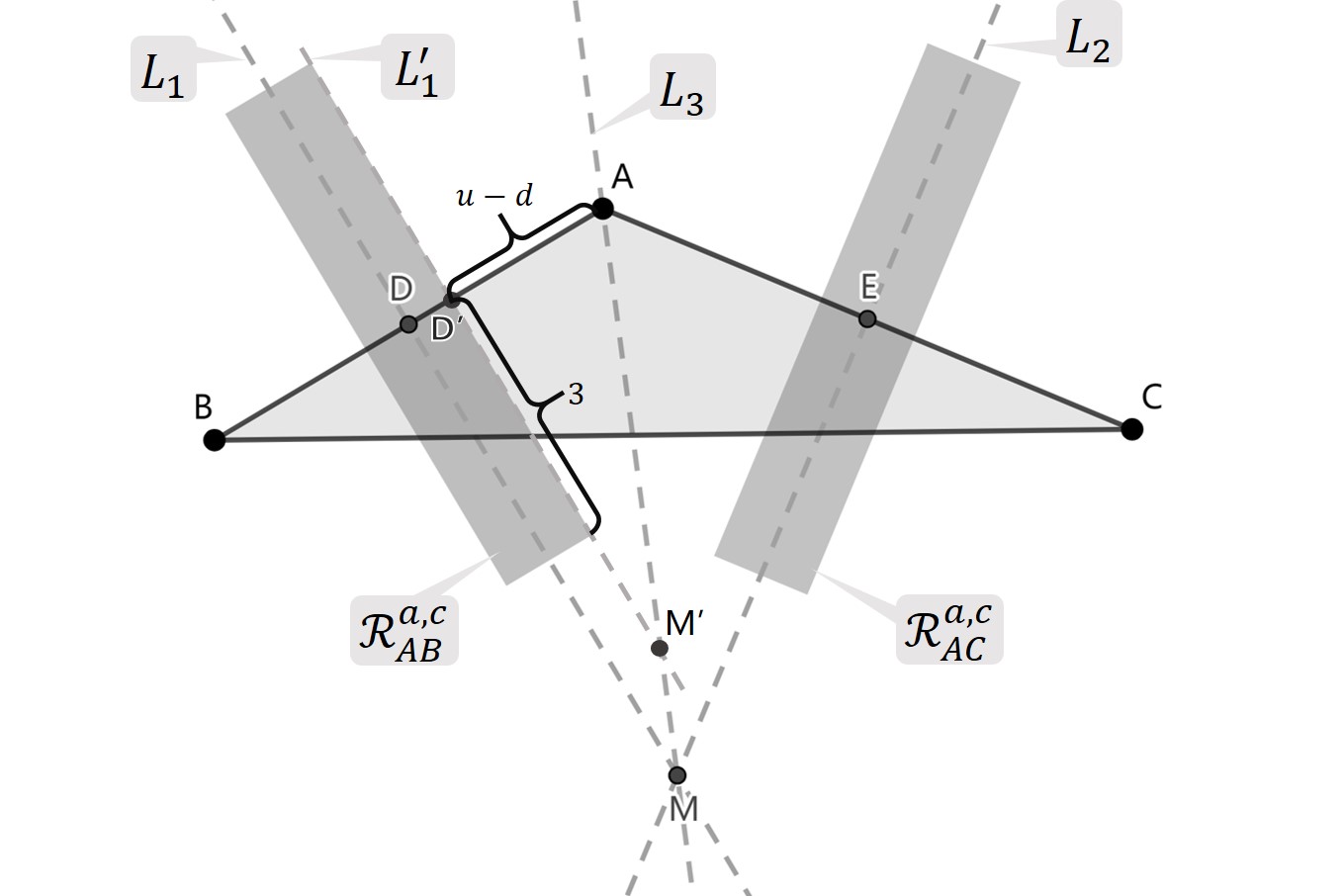}
\caption  {An illustration for $D'$, $M'$ and $L_1'$.}
\label{f0207}
\end{figure}
  We observe that
  \begin{equation}\label{angle}
    \angle D'M'A = \angle DMA \mbox{ and } \tan \angle D'M'A = \frac{\|A-D'\|}{\|D'-M'\|} = \frac{u-d}{\|D'-M'\|}
  \end{equation}
  where in the last inequality we recall that $\|A-D'\|= \|A-D\|-\|D-D'\|$, $\|A-D\|=u$ and $\|D-D'\|=d$.
  Combining \eqref{suff3} and \eqref{angle}, we deduce that
  $$  \frac{u-d}{\|D'-M'\|} \overset{\eqref{angle}}{=} \tan \angle D'M'A \overset{\eqref{suff3}}{<} \frac{u-d}{3},  $$
  which implies
  $$ \|D'-M'\| > 3 . $$
  This, combined with the fact that half of the length of the long sides of $\mathcal{R}_{AB}^{a,c}$ is 3, shows that
  $$ \mathcal{R}_{AB}^{a,c} \cap L_3 = \emptyset. $$
  By a similar argument, we also have $ \mathcal{R}_{AC}^{a,c} \cap L_3 = \emptyset$ with the help of \eqref{suff3}.
  This shows that \eqref{suff1} is true and hence \eqref{suff01} holds.

\bigskip
\emph{Case 3:} $\angle A \in [\pi/3, \pi- \arctan (2c/9)]$. In this case, $W(b)$ may not be empty. Now, we assume that $W(b) \ne \emptyset$, which implies that
$ \mathcal{R}_{AB}^{a,c} \cap \mathcal{R}_{AC}^{a,c} \ne \emptyset$. Moreover, denote by $\mathcal{V}^{d}_{L_i}$ the closed $d$-neighbourhood of lines $L_i$, $i=1,2$. Then $\mathcal{V}^{d}_{L_1} \cap \mathcal{V}^{d}_{L_2}$ is a rhombus $\mathcal{T}_M$ centered at $M$ satisfying $\mathcal{R}_{AB}^{a,c} \cap \mathcal{R}_{AC}^{a,c} \subset \mathcal{T}_M$.
We will show that
\begin{equation}\label{geom40}
  \text{diam}\,  \mathcal{T}_M \le 324\frac{a}{c^2}.
\end{equation}
See Figure \ref{f0204} for an illustration.
 \vspace*{7pt}
\begin{figure}[h]
\hspace{-0.7cm}
\centering
\includegraphics[width=10cm]{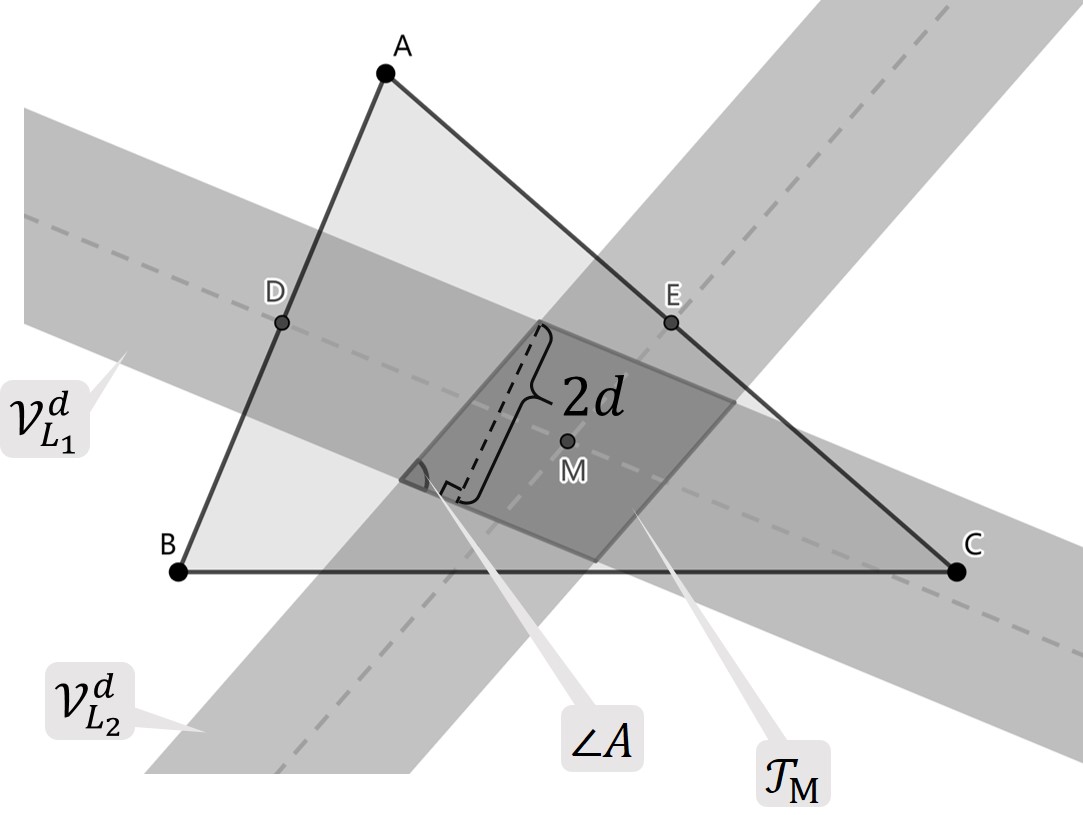}
\caption  {An illustration for the estimate $\|  y -M\|$.}
\label{f0204}
\end{figure}

Denote the length of two diagonals of $\mathcal{T}_M$ by $d_1$ and $d_2$ and the the length of four sides of $\mathcal{T}_M$ by $l$.
We have
\begin{equation} \label{geom41}
\text{diam}\, \mathcal{T}_M = \max \{d_1,d_2\},
 \end{equation}
\begin{equation} \label{geom42}
d_1^2+d_2^2 = 4l^2
 \end{equation}
and
\begin{equation} \label{geom43}
 l = \frac{2d}{\sin \angle A}.
 \end{equation}
Since $\angle A \in [\pi/3, \pi- \arctan (2c/9)]$, we have
\begin{equation} \label{geom44}
 \sin \angle A \ge \sin(\arctan \frac{2c}{9}) \ge \sin   \frac{c}9 \ge \frac{c}{18}
 \end{equation}
 where in the second last inequality we use the fact that $\arctan y > \frac {y}2$ if $0<y<1$ and in the last inequality we use the fact that $\sin y > \frac {y}2$ if $0<y<1$.

Combining \eqref{geom41}, \eqref{geom42},  \eqref{geom43} and \eqref{geom44},
we obtain
\begin{equation} \label{geom45}
\text{diam}\, \mathcal{T}_M \le 2l \le \frac{72d}{c} = 324\frac{a}{c^2}
 \end{equation}
where in the last equality we recall $d= \frac{9}{2} \frac{a}c$. Therefore, we conclude \eqref{geom40}.

Combining \emph{Case 2} and \emph{Case 3}, we conclude \eqref{geo20}.

\medskip
Finally, we show \eqref{geo3}.
Let $x \in W(b)$. By \eqref{geo1} and \eqref{geom41}, we have
$$ |b-h| = |b-\|M-A\|| \le |b-\|x-A\|| + \|x-M\| \lesssim a+ \frac{ a}{c^2}\lesssim \frac{ a}{c^2}. $$
The proof is complete.
\end{proof}

Now, we are in a position to show:

\begin{proof}[Proof of Lemma \ref{geom3}]
    For $b \in [\frac12,2]$, define
    \begin{align*}
\wz W(b)& := W(b) \times \{b\} \\
  & = \left \{ \begin{array}{ll}
  \quad & \quad b-a \le \|x-A\| \le b+a,  \\
 (x_1,x_2,x_3)=(x,x_3) \in  \rr^3 :  &  \quad b-a \le \|x-B\| \le b+a, \ x_3=b \\ 
   \quad  & \quad b-a \le \|x-C\| \le b+a, 
\end{array}
\right \}.
\end{align*}

First we assume $\triangle ABC$ degenerates. Then by \eqref{g30}, we know $$\wz W(b) = \emptyset \text{ for all $b \in [\frac12,2]$}.$$
Hence the lemma holds for this case.

\medskip
Next, we assume $\triangle ABC$ is non-degenerate.
Then by \eqref{geo20}, we have
\begin{equation}\label{geo21}
        \wz W(b) \subset B((M,b), K\frac{ a}{c^2}) \cap \{x_3=b\}\subset \rr^3, \text{ for all $b \in [\frac12,2]$},
    \end{equation}
    where $M$ is the circumcenter of the triangle $\triangle ABC$.

Since $\wz W(b) \ne \emptyset$ implies $ h- K\frac{ a}{c^2}\le b \le h+K\frac{ a}{c^2}$ from Lemma \ref{geom2}, we know
\begin{equation}\label{geo22}
  W \subset \bigcup_{\{b \, | \, \wz W(b) \ne \emptyset \} }\wz W(b) \subset \bigcup_{b \in [h- K\frac{ a}{c^2},h+K\frac{ a}{c^2}]}\wz W(b).
\end{equation}
 Then combining \eqref{geo21} and \eqref{geo22}, we deduce
\eqref{geo23}, i.e.
$$  \text{diam} \, W \lesssim \frac{ a}{c^2},$$
 which finishes the proof.
\end{proof}

\noindent Jiayin Liu

\noindent Department of Mathematics and Statistics, University of Jyv\"{a}skyl\"{a}, P.O. Box 35
(MaD), FI-40014, Jyv\"{a}skyl\"{a}, Finland

\noindent{\it E-mail }:  \texttt{jiayin.mat.liu@jyu.fi}

\end{document}